\documentclass[11pt,reqno,tbtags]{amsart}

\usepackage{mathtools}
\usepackage{amsthm,amssymb,amsfonts,amsmath}
\usepackage{amscd} 
\usepackage{mathrsfs}
\usepackage[numbers, sort&compress]{natbib} 
\usepackage{subfigure, graphicx}
\usepackage{vmargin}
\usepackage{setspace}
\usepackage{paralist}
\usepackage[pagebackref=false]{hyperref}
\usepackage{color,xcolor}
\usepackage{stmaryrd} 
\usepackage{framed,mdframed,color}
\mdfdefinestyle{thmbox}{leftmargin=0cm,rightmargin=0cm,innerleftmargin=0.3cm,innertopmargin=0.2cm,innerrightmargin=1cm,roundcorner=10pt}
\mdfapptodefinestyle{thmbox}{backgroundcolor=black!10,linecolor=red!40!black,linewidth=1.5pt}
\providecommand{\R}{}

\providecommand{\N}{}

\renewcommand{\R}{\mathbb{R}}

\renewcommand{\N}{{\mathbb N}}

\newcommand{\E}[1]{{\mathbf E}\left(#1\right)}										
\newcommand{\e}{{\mathbf E}}
\newcommand{\V}[1]{{\mathbf{Var}}\left\{#1\right\}}

\newcommand{\p}[1]{{\mathbf P}\left\{#1\right\}}

\newcommand{\I}[1]{{\mathbf 1}_{[#1]}}
\newcommand{\set}[1]{\left\{ #1 \right\}}
\newcommand{\Cprob}[2]{\mathbf{P}\set{\left. #1 \; \right| \; #2}} 
\newcommand{\probC}[2]{\mathbf{P}\set{#1 \; \left|  \; #2 \right. }}

\newcommand{\Cexp}[2]{\mathbf{E}\set{\left. #1 \; \right| \; #2}}

\newcommand\cB{\mathcal B}

\newcommand\cG{\mathcal G}
\newcommand\cH{\mathcal H}

\newcommand\cL{{\mathcal L}}
\newcommand\cM{\mathcal M}
\newcommand\cN{\mathcal N}

\newcommand\cS{{\mathcal S}}
\newcommand\cT{{\mathcal T}}
\newcommand\cU{{\mathcal U}}

\newcommand{\convdist}{\ensuremath{\stackrel{\mathrm{d}}{\rightarrow}}}

\newcommand{\convp}{\ensuremath{\stackrel{\mathrm{prob}}{\longrightarrow}}}

\newcommand\sumstar{\mathop{{\sum\nolimits^{\mathrlap{\star}}}}}
\newcommand{\pran}[1]{\left(#1\right)}
\providecommand{\eps}{}
\renewcommand{\eps}{\epsilon}
\providecommand{\ora}[1]{}
\renewcommand{\ora}[1]{\overrightarrow{#1}}
\newcommand\urladdrx[1]{{\urladdr{\def~{{\tiny$\sim$}}#1}}} 
\DeclareRobustCommand{\SkipTocEntry}[5]{} 

\reversemarginpar
\definecolor{clou}{rgb}{0.8,0.25,0.5125}

\newtheorem{thm}{Theorem}
\newtheorem{lem}[thm]{Lemma}
\newtheorem{prop}[thm]{Proposition}
\newtheorem{cor}[thm]{Corollary}

\newtheorem{fact}[thm]{Fact}

\newcommand{\dseq}{\mathrm{d}}
\newcommand{\cseq}{\mathrm{c}}

\newcommand{\dist}{\mathrm{dist}}
\newcommand{\rK}{\mathrm{K}}

\newcommand{\rH}{\mathrm{H}}
\newcommand{\tdeg}{a}
\newcommand{\gdeg}{b}
\newcommand{\urand}{\mathrm{U}}
\numberwithin{thm}{section}

\newcommand{\binprob}{\rho}

\providecommand{\parent}{}
\renewcommand{\parent}{\mathrm{par}}
\newcommand{\vertices}{\mathrm{v}}
\newcommand{\edges}{\mathrm{e}}

\begin{document}

\title{Random tree-weighted graphs} 
\author{Louigi Addario-Berry}
\author{Jordan Barrett}
\address{Department of Mathematics and Statistics, McGill University, Montr\'eal, Canada}
\email{louigi.addario@mcgill.ca}
\email{jordan.barrett@mail.mcgill.ca}
\date{August 27, 2020; revised January 22, 2021} 
\urladdrx{http://problab.ca/louigi/}


\begin{abstract} 
For each $n \ge 1$, let $\dseq^n=(d^{n}(i),1 \le i \le n)$ be a sequence of positive integers with even sum $\sum_{i=1}^n d^n(i) \ge 2n$. Let $(G_n,T_n,\Gamma_n)$ be uniformly distributed over the set of simple graphs $G_n$ with degree sequence $\dseq^n$, endowed with a spanning tree $T_n$ and rooted along an oriented edge $\Gamma_n$ of $G_n$ which is not an edge of $T_n$. Under a finite variance assumption on degrees in $G_n$, we show that, after rescaling, $T_n$ converges in distribution to the Brownian continuum random tree as $n \to \infty$. Our main tool is a new version of Pitman's additive coalescent \cite{MR1673928}, which can be used to build both random trees with a fixed degree sequence, and random tree-weighted graphs with a fixed degree sequence. As an input to the proof, we also derive a Poisson approximation theorem for the number of loops and multiple edges in the superposition of a fixed graph and a random graph with a given degree sequence sampled according to the configuration model; we find this to be of independent interest. 
\end{abstract}

\maketitle



\section{\bf Introduction}\label{sec:intro} 

By a {\em rooted tree} we mean a labeled tree $t=(\vertices(t),\edges(t))$, with a distinguished root node denoted $r(t)$. 
A {\em tree-rooted graph} is a pair $(g,t,\gamma)$ where $g=(\vertices(g),\edges(g))$ is a labeled graph, $t=(\vertices(t),\edges(t))$ is a spanning tree of $g$, and $\gamma=uv$ is a distinguished oriented edge with $\{u,v\} \in \edges(g)\setminus \edges(t)$. We view $t$ as a rooted tree by setting $r(t)=u$. 

Throughout this work, we allow our graphs to have multiple edges and loops; in tree-rooted graphs, the root edge is allowed to be a loop. We say $(g,t,\gamma)$ is {\em simple} if $g$ is simple, i.e., if $g$ contains no multiple edges or loops. 

For a node $u$ of a rooted tree $t$, we write $c_t(u)$ for the number of children of $u$ in $t$. Given a rooted tree $t$ with $\vertices(t)=[n]:=\{1,2,\ldots,n\}$, the {\em child sequence} of $t$ is the sequence $\cseq_t=(c_t(i),1 \le i \le n)$. Similarly, given a tree-rooted graph $(g,t,\gamma)$ with vertex set $\vertices(g)=[n]$, the {\em degree sequence} of $(g,t,\gamma)$ is the sequence $(d_g(i),1 \le i \le n)$, where $d_g(i)$ is the number of endpoints of edges incident to $i$ in $g$; here loops are counted twice. 

For any sequence $\dseq=(d(1),\ldots,d(n))$ of non-negative integers, we define the {\em degree distribution} 
$p_{\dseq}=(p_{\dseq}(k),k \ge 1)$ of $\dseq$ by letting $p_{\dseq}(k) = \#\{i \in [n]: d(i)=k\}/n$. 

The following theorem contains our main result, which is an invariance principle for the spanning trees in random tree-rooted graphs with a fixed degree sequence. To state it, two further pieces of notation are needed. Given a finite graph $g=(v,e)$ and a constant $c >0$, we write $cg$ for the measured metric space $(v,\dist,\pi)$ whose points are the elements of $v$, with $\dist(x,y):= c\cdot \dist_g(x,y)$, where $\dist_g(x,y)$ denotes graph distance in $g$, and with $\pi$ the uniform probability measure on $v$. Also, for a sequence $p=(p(k),k \ge 1)$ of real numbers, we write $\mu_1(p):=\sum_{k\ge 1} kp(k)$ and $\mu_2(p):= \sum_{k \ge 1} k^2 p(k)$.
\begin{thm}\label{thm:main}%
For each $n \ge 1$ let $\dseq^{n}=(d^{n}(i),1 \le i \le n)$ be a degree sequence with $\min_{1 \le i \le n} d^n(i) \ge 1$, with $\sum_{i \in [n]} d^n(i) \ge 2n$ and with $\sum_{i \in [n]} d^n(i)$ even. Let $p^n$ be the degree distribution of $\dseq^n$. Suppose that there exists a probability distribution $p=(p(k),k \ge 1)$ such that (a) $p^n \to p$ pointwise and $p(2)<1$, and (b) $\mu_2(p^n) \to \mu_2(p) \in (0,\infty)$. Then there exists $\sigma=\sigma(p) \in (0,\infty)$ such that the following holds. 

For $n \ge 1$ let $(G_n,T_n,\Gamma_n)$ be chosen uniformly at random among all simple tree-rooted graphs with vertex set $[n]$ and degree sequence  $d^n$. Then
\[
\frac{\sigma}{n^{1/2}} T_n \convdist \cT
\]
as $n \to \infty$ with respect to the Gromov-Hausdorff-Prokhorov topology, where $\cT$ is the Brownian continuum random tree. 
\end{thm}
We refer the reader to \cite{MR3035742} for a good discussion of the Gromov-Hausdorff-Prokhorov topology aimed at probabilists. 
The technical insight underlying the proof of Theorem~\ref{thm:main} is the fact that  Pitman's additive coalescent~\cite{MR1673928} can be modified to yield a simple construction procedure for random tree-weighted graphs with a given degree sequence. We anticipate that this procedure has further interesting features to be explored. 

\subsection{Related work}
The enumerative combinatorics of tree-rooted maps was developed in the 1960's and 1970's \cite{MR205882,MR314687}. The area has seen renewed attention over the last decade or so \cite{MR3366469,MR2438581,MR2285813}. 
Random tree-rooted maps can be interpreted as samples from a Fortuin-Kastelyn model at zero temperature, and are an active object of study in the planar probability community (see, e.g.,\cite{MR4126936,MR3861296,MR4010949,MR3778352,MR3681382,gwynne2017,li2017}). 

There has also been some work on the typical number of spanning trees in uniformly random graphs \cite{MR657198,MR3645782,MR3177540} with given degree sequences. (In such models, the underlying graph is sampled uniformly at random from some set of allowed graphs; in our model, it is the tree-weighted graph which is uniformly random, which means the underlying measure on graphs is biased in favour of graphs with a greater number of spanning trees.) 

Except in the setting of graphs on surfaces, we have not found any previous work on tree-weighted graphs, random or otherwise.

\subsection{Overview of the proof}
We begin with a small number of facts and definitions that are required for the overview. We say a sequence $\cseq=(c(i),1 \le i \le n)$ of non-negative integers is a child sequence if it is the child sequence of some tree. Note that $\cseq$ is a child sequence if and only if $\sum_{1 \le i \le n} c(i)=n-1$, in which case 
\begin{equation}\label{eq:treecount}
\#\{\mbox{rooted trees}~t: \cseq_t=\cseq\} 
= {n-1 \choose c(1), \dots ,c(n)} = \frac{1}{n} \frac{n!}{\prod_{i=1}^n c(i)!}\, ;
\end{equation}
see \cite{MR0274333}, Section 3.3.
 
Given any sequence $\mathrm{c}=(c(i),1 \le i \le n)$ of non-negative integers, for $k \ge 1$ we write $Q_{\mathrm{c}}(k) = \#\{i \in [n]: c(i)=k\}$. We call $Q_{\mathrm{c}}=(Q_{\mathrm{c}}(k),k \ge 0)$ the {\em child statistics vector} of $\cseq$. For a tree $t$ with child sequence $\cseq_t$, we will  sometimes write $Q_t=Q_{\cseq_t}$ for succinctness. 

Given a graph $g$, for an edge $e \in \edges(g)$ we write $m_g(e)$ for the multiplicity of edge $e$ in $g$. Given a degree sequence $\dseq=(d(1),\ldots,d(n))$, the classical {\em configuration model} \cite[Chapter 7]{MR3617364} produces a random graph $G$ such that 
for any fixed graph $g$ with degree sequence $\dseq$, 
\begin{equation}\label{eq:cm}
\p{G=g} \propto \frac{1}{\prod_{i=1}^n 2^{m_g(ii)} \prod_{e \in \edges(g)} m_g(e)!}\, .
\end{equation}
In Section~\ref{sec:additive_coal}, we define a sampling procedure, inspired by the configuration model and by Pitman's additive coalescent~\cite{MR1673928}, which produces a random tree-weighted graph $(G,T,\Gamma)$ with the property that for any fixed tree-weighted graph $(g,t,\gamma)$ with degree sequence $\dseq$, 
\begin{equation}\label{eq:twglaw}
\p{(G,T,\Gamma)=(g,t,\gamma)} \propto \frac{2^{\I{\gamma~\mathrm{is~a~loop}}}\cdot m_{g-t}(\gamma)}{\prod_{i=1}^n 2^{m_{g-t}(ii)} \cdot \prod_{e \in \edges(g)} m_{g-t}(e)!}\, ,
\end{equation}
where $g-t$ is the graph with the same vertex set as $g$ and with edge multiplicities given by
\[
m_{g-t}(e) = \begin{cases}
				m_g(e)	& \mbox{ if } e \not \in \edges(t) \\
				m_g(e)-1& \mbox{ if }e \in \edges(t)\, .
			\end{cases}
\]

We call a random tree-weighted graph $(G,T,\Gamma)$ with distribution given by (\ref{eq:twglaw}) a {\em random tree-weighted graph with degree sequence $\dseq$}. Note that in this case, conditionally given that $G$ is simple, $(G,T,\Gamma)$ is uniformly distributed over simple tree-rooted graphs with degree sequence $\dseq$.

The sampling procedure we use has enough exchangeability that, conditional on its child sequence, the resulting spanning tree $T$ is uniformly distributed; that is, for any fixed child sequence $\cseq=(c(i),1 \le i \le n)$, and any tree $t$ with $\cseq_t=\cseq$, 
\[
\probC{T=t}{\cseq_T=\cseq} = {n-1 \choose 
c(1),\dots,c(n)}^{-1}\, .
\]

Now let $(\dseq^n,n \ge 1)$ be a sequence of degree sequences satisfying the conditions of Theorem~\ref{thm:main}; for each $n$ let $(G(\dseq^n),T(\dseq^n),\Gamma(\dseq^n))$ be a random tree-weighted graph with degree sequence $\dseq^n$. 
We prove (see Proposition~\ref{prop:degree_concentration}) that 
there is a probability distribution $q=(q(i),i \ge 0)$ with $\mu_2(q)<\infty$ such that the child statistics vector $Q_{T(\dseq^n)}$ satisfies that
$n^{-1}Q_{T(\dseq^n)}(\tdeg) \to q(\tdeg)$ in probability for all $\tdeg \ge 0$, and moreover that $\mu_2(n^{-1}Q_{T(\dseq^n)}) \to \mu_2(q)$ in probability. It then follows from a result of Broutin and Marckert \cite{MR3188597} that 
\begin{equation}\label{eq:weak_conv}
\frac{\sigma}{n^{1/2}} T(\dseq^n) \convdist \cT
\end{equation}
in the Gromov-Hausdorff-Prokhorov sense, where $\sigma^2 = \mu_2(q)-1$ and $\cT$ is the Brownian continuum random tree.

This is not quite the convergence claimed in Theorem~\ref{thm:main}, because 
$(G(\dseq^n),T(\dseq^n),\Gamma(\dseq^n))$ is a random tree-weighted graph with degree sequence $\dseq^n$, whereas Theorem~\ref{thm:main} concerns random {\em simple} tree-weighted graphs. To obtain Theorem~\ref{thm:main} from (\ref{eq:weak_conv}), we show that there is $\alpha \in (0,1]$ such that as $n \to \infty$, 
\begin{equation}\label{eq:simplicity_constant}
\probC{G(\dseq^n)\mbox{ is simple}}{T(\dseq^n)} \to \alpha
\end{equation}
in probability. The value of (\ref{eq:simplicity_constant}), informally, is that it implies that conditioning $G(\dseq^n)$ to be simple has an asymptotically negligible effect on the law of $T(\dseq^n)$. Since the law of $(G(\dseq^n),T(\dseq^n),\Gamma(\dseq^n))$, conditional on the simplicity of $G(\dseq^n)$, is uniform over simple tree-weighted graphs with degree sequence $\dseq^n$, we can then conclude straightforwardly. 

We finish the overview with a brief discussion of how we prove (\ref{eq:simplicity_constant}).
Our procedure for constructing random tree-weighted graphs with a given degree sequence first constructs the tree $T(\dseq^n)$, then randomly pairs the remaining half-edges as in the standard configuration model. Viewing $T(\dseq^n)$ as fixed, this leads us to the following more general question. Let $G=(V,E)$ be a random graph with a given degree sequence generated according to the configuration model, and let $T=(V,E')$ be a fixed, simple graph with the same vertex set. What is the probability that the union of $G$ and $T$ forms a simple graph (i.e. that $G$ is a simple graph and that $E$ and $E'$ are disjoint)? We provide a partial answer to this question by proving a fairly general Poisson approximation theorem for the number of loops and multiple edges in the superposition of a fixed graph and a random graph drawn from the configuration model (see Theorem~\ref{Poisson Asymptotics}). 
In order to apply Theorem~\ref{Poisson Asymptotics}, we need that the joint degree statistics in $G(\dseq^n)$ and in $T(\dseq^n))$ are sufficiently well-behaved; proving this is the task of Section~\ref{sec:concentration}. We then state and prove Theorem~\ref{Poisson Asymptotics} in Section~\ref{sec:poisson_approx}, and finally put all the pieces together to prove Theorem~\ref{thm:main} in Section~\ref{sec:finalproof}.
\begin{figure}[htb]
\includegraphics[width=0.9\textwidth]{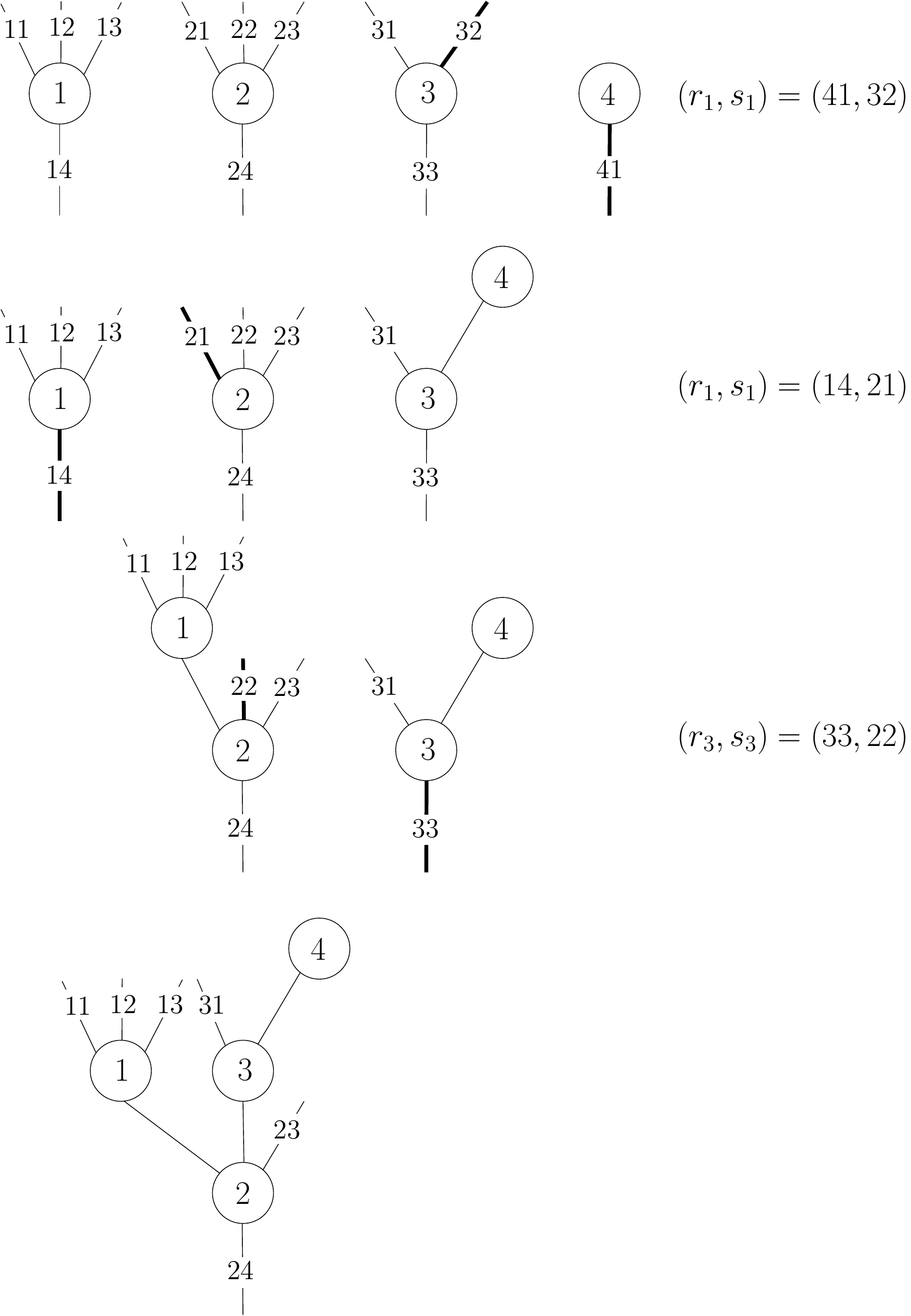}
\caption{An example of an execution path of Pitman's additive coalescent. The forests $F_1,F_2,F_3$ and $F_4$ are displayed in successive rows.}
\label{fig:coalescent_example}
\end{figure}
\medskip

\section{Pitman's additive coalescent with a fixed degree sequence}\label{sec:additive_coal}

\subsection{The sampling process}
Let $\dseq=(d(1),\ldots,d(n))$ be a degree sequence, which is to say that $d(1),\ldots,d(n)$ are non-negative integers. 
To be well-defined, the next process requires that $\sum_{1 \le i \le n} d(i) \ge 2n-1$ and that $d(i) \ge 1$ for all $1 \le i \le n$. 
{
\begin{mdframed}[style=thmbox]
{\bf Pitman's additive coalescent.}
The process has $n-1$ steps, and at the start of step $k$ consists of a rooted forest $F_k(\dseq)=\{T_1^{k}(\dseq),\ldots,T_{n+1-k}^{k}(\dseq)\}$ with $n+1-k$ trees. At the start of step $1$, these trees are isolated vertices with labels $1,\ldots,n$.  Vertex $i$ has $d(i)$ half-edges $(i1,i2,\ldots,id(i))$ attached to it, and $id(i)$ is distinguished as the {\em root half-edge}.

\noindent {\em Step k}: 

Choose a uniformly random pair $(r_k,s_k)$, where $r_k$ is a root half-edge which is not paired in $F_k(\dseq)$ and $s_k$ is a non-root half-edge which is not paired in $F_k(\dseq)$ and additionally belongs to a different tree of $F_k(\dseq)$ from $r_k$. 

Pair the half-edges $r_k$ and $s_k$ to create an edge $e_k$ connecting their endpoints; this merges two trees of $F_k(\dseq)$. 
The root of the new tree is the same as the root of the tree of $F_k(\dseq)$ containing $s_k$. In the new tree, the vertex incident to $r_k$ is the child of the vertex incident to $s_k$. 

Define $F_{k+1}(\dseq)$ to be the forest consisting of the new tree thus created, together with the remaining $n-k-1$ unaltered trees of $F_k(\dseq)$. 
\end{mdframed}
}
\medskip 

An example is shown in Figure~\ref{fig:coalescent_example}.
Write $T(\dseq) = T^n_1(\dseq)$ for the single tree in the random forest $F_n(\dseq)$. Attached to the tree $T(\dseq)$ there is a single pendant (unpaired) root half-edge which is incident to the root of $T(\dseq)$, and if $\sum_{i=1}^n d(i) > 2n-1$ then there are also other pendant half-edges. By ignoring pendant half-edges, we may view $T(\dseq)$ as a random rooted tree with vertex set $[n]$. 

It will be useful to additionally define two edge labellings of $T(\dseq)$, denoted $\rK$ and $\rH$. We define $\rK(e)$ to be the step at which edge $e$ was added; so $\rK(e_k):= k$. We define $\rH(e)$ to be the non-root half-edge used in creating $e$; so  $\rH(e_k)=s_k$. Note that $\rK$ is a bijection between $\edges(T(\dseq))$ and $[n-1]$. Also, if $i \in [n]$ has $c_{T(\dseq)}(i)=c$, then $\rH$ assigns $c$ distinct half-edges from the set $\{i1,\ldots,i(d(i)-1)\}$ to the edges between $i$ and its children in $T(\dseq)$. 

We use the phrase ``execution path'' to mean a sequence of pairs $(r_1,s_1),\ldots,(r_{n-1},s_{n-1})$ which may concievably appear as the ordered sequence of pairs of half-edges added during the course of Pitman's coalescent. 

The next proposition fully describes the joint distribution of $T(\dseq)$, $\rK$, and $\rH$. 
In its proof, and in what follows, for a rooted tree $t$ and a node $u \in \vertices(t)\setminus\{r(t)\}$ we write $\parent(u)$ for the parent of $u$ in $t$. Also, we use the falling factorial notation $(k)_\ell:=k(k-1)\cdot\ldots\cdot (k-\ell+1)=k!/(k-\ell)!$. 
\begin{prop}\label{prop:tree_formula} 
Let $\dseq=(d(1),\ldots,d(n))$ be a degree sequence with $d(i) \ge 1$ for all $i \in [n]$ and with $\sum_{i=1}^n d(i) \ge 2(n-1)$, and write $m=\frac{1}{2} \sum_{i=1}^n d(i)$. (Note: we allow that $\sum_{i=1}^n d(i)$ is odd.) Then the following properties all hold. 
\begin{enumerate}
\item For any fixed rooted tree $t$ with vertex set $[n]$, 
\[
\p{T(\dseq)=t} = 
\frac{1}{(2m-n)_{n-1}} \prod_{i=1}^n (d(i)-1)_{c_t(i)}\, .
\]
\item Fix any set $\cH \subset \bigcup_{i=1}^n \{i1,\ldots,i(d(i)-1)\}$ with $|\cH|=n-1$. 
Conditionally given that $\{s_1,\ldots,s_{n-1}\}=\cH$, the triple 
$(T(\dseq),\rK,\rH)$ is uniformly distributed over the $((n-1)!)^2$  triples which are consistent with the event $\{s_1,\ldots,s_{n-1}\}=\cH$. 
\item The sequence $(s_1,\ldots,s_{n-1})$ of non-root half-edges, added by Pitman's coalescent, is uniformly distributed over the set of sequences of $(n-1)$ distinct elements of $\bigcup_{1 \le i \le n} \{i1,\ldots,i(d(i)-1)\}$. Consequently, $\{s_1,\ldots,s_{n-1}\}$ is a uniformly random size-$(n-1)$ subset of $\bigcup_{1 \le i \le n} \{i1,\ldots,i(d(i)-1)\}$. 
\item 
Finally, conditionally given that $T(\dseq)=t$ and given the set $\{s_1,\ldots,s_{n-1}\}$ of non-root half-edges added by Pitman's coalescent, the ordering $(e_1,\ldots,e_{n-1})$ of $\edges(t)$ is uniformly distributed over the $(n-1)!$ possible orderings of $\edges(t)$. 
\end{enumerate}
\end{prop}
\begin{proof}
At step $i$ of the process, there are $n+1-i$ components and $2m-n+1-i$ unpaired non-root half-edges. We may specify the pair $(r_i,s_i)$ by first revealing the non-root half-edge $s_i$, then revealing $r_i$. Whatever the choice of $s_i$, there are $n-i$ possibilities for $r_i$, so the number of distinct choices for the pair $(r_i,s_i)$ is $(2m-n+1-i) (n-i)$. 
Thus, the total number of possible execution paths for the process is 
\begin{equation}\label{eq:execution_path_count}
\prod_{i=1}^{n-1} (2m-n+1-i) (n-i) = (n-1)! (2m-n)_{n-1}. 
\end{equation}

The execution path followed by the process is uniquely determined by the tree $T(\dseq)$ and the functions $\rK:\edges(T(\dseq)) \to [n-1]$ and $\rH:\edges(T(\dseq)) \to \bigcup_{i=1}^n \{i1,\ldots,i(\dseq_i-1)\}$. To see this, fix any $k \in [n-1]$. Then the edge $e_k$ created at step $k$ of Pitman's coalescent may be recovered as $e_k = \rK^{-1}(k)$; and, if $e_k=uv$ with $v=\parent(u)$ then the half-edges paired to create $e_k$ are the root half-edge $vd(v)$ incident to $v$ and the half-edge $\rH^{-1}(e_k)$. 

Now, fix any tree $t$ with degree sequence $\dseq$, any bijection $\mathrm{k}:\edges(t) \to [n-1]$, and any function $\mathrm{h}:\edges(t) \to \N$ which, for all $i \in [n]$, assigns $c_t(i)$ distinct values from the set $\{1,\ldots,(d(i)-1)\}$ to the edges between $i$ and its children in $t$. Together with (\ref{eq:execution_path_count}), the observation of the preceding paragraph implies that 
\[
\p{T(\dseq)=t,\rK=\mathrm{k},\rH=\mathrm{h}} = \frac{1}{(n-1)! (2m-n)_{n-1}}\, .
\]

Having fixed the tree $t$, the number of possible values for $\rK$ is $(n-1)!$ and the number of possible values for $\rH$ is $\prod_{i \in [n]} (d(i)-1)_{c_t(i)}$. It follows that 
\[
\p{T(\dseq)=t} = \frac{(n-1)!\cdot \prod_{i \in [n]} (d(i)-1)_{c_t(i)}}{(n-1)! (2m-n)_{n-1}} = 
\frac{\prod_{i \in [n]} (d(i)-1)_{c_t(i)}}{(2m-n)_{n-1}}
\, ,
\]
which proves the first claim of the proposition. 

Next, fix $\cH$ as in the second assertion of the proposition, 
and any ordering of $\cH$ as $(h_1,\ldots,h_{n-1})$. Then the number of execution paths which yield that $s_k=h_k$ for $k \in [n-1]$ is precisely $(n-1)!$. To see this, note that if $s_j=h_j$ for $1 \le j \le k$ then, whatever the choices of the root half-edges $(r_j,1 \le j \le k)$, the forest $F^n_k$ has $n+1-k$ component trees so there are $n-k$ unpaired root half-edges in components different from that of $s_k$; any such root half-edge may be chosen as $r_k$. Since there are also $(n-1)!$ possible orderings of $\cH$, the second assertion of the proposition follows. 

To prove the third statement, fix a set $\cH$ and an ordering $(h_1,\ldots,h_{n-1})$ of its elements, as in the previous paragraph. 
For each $1 \le k < n-1$, given that $s_j=h_j$ for $1 \le j < k$, whatever the choices of $(r_j,1 \le j < k)$ may be, there are $n-k$ ways to choose $r_k$ in a distinct tree from $h_k$. It follows that there are $\prod_{k=1}^{n-2} (n-k) = (n-1)!$ execution paths with the property that $s_k=h_k$ for each $1 \le k \le n-1$. 
Since this number does not depend on $\cH$, it follows that each size-$(n-1)$ subset of $\bigcup_{1 \le i \le n} \{i1,\ldots,i(d(i)-1)\}$ is equally likely. 

Finally, fix both the tree $t$ and an unordered set 
$\cH$ of non-root half-edges with $|\cH \cap \{i1,\ldots,i(d(i)-1)\}|=c_t(i)$ for all $i \in [n]$. We consider the number of execution paths which yield $T(\dseq)=t$ and $\{s_1,\ldots,s_{n-1}\}=\cH$. The number of choices of an ordering function $\mathrm{k}:\edges(t) \to [n-1]$ consistent with these constraints is still $(n-1)!$. Moreover, whatever the choice of $\mathrm{k}$, under the further constraint $\rK=\mathrm{k}$, the number of possibilities for $\rH$ is $\prod_{i \in [n]} c_t(i)!$. To see this, note that for each $i \in [n]$, the constraints precisely imply that   $\cH \cap \{i1,\ldots,i(d(i)-1)\}=\{s_1,\ldots,s_{n-1}\} \cap \{i1,\ldots,i(d(i)-1)\}$, and $\rH$ is fixed once we additionally specify which of these $c_t(i)$ half-edges is matched to which child of $i$, for each $i \in [n]$. It follows that the number of execution paths which yield that $T(\dseq)=t$, that $\rK=\mathrm{k}$ and that $\{s_1,\ldots,s_n\}=\cH$ is 
\[
\prod_{i=1}^n c_t(i)!\, . 
\]
As this quantity doesn't depend on the choice of the ordering function $\mathrm{k}$, the final assertion of the proposition follows. 
\end{proof}

We state a corollary of the above proposition, for later use. 
\begin{cor}
The tree $T(\dseq)$ is a uniformly random rooted tree with child sequence $\cseq_T$.
\end{cor}
The corollary follows since the formula for $\p{T(\dseq)=t}$ from Proposition~\ref{prop:tree_formula}  only depends on $t$ through $\cseq_t$. 

We now assume that $\sum_{i=1}^n d(i) \ge 2n$ and that $\sum_{i=1}^n d(i)$ is even, and define a random tree-rooted graph $(G,T,\Gamma)=(G(\dseq),T(\dseq),\Gamma(\dseq))$ as follows: First, let $T=T(\dseq)$ be the random tree built by Pitman's coalescent, and let $\Gamma^+=\Gamma^+(\dseq)$ be its root half-edge. We refer to $T$ as the {\em spanning tree-elect} of a to-be-constructed tree-rooted graph. 
Next, choose a uniformly random matching of the $2m-2(n-1)$ pendant half-edges attached to $T$, and pair the half-edges according to this matching to create $G=G(\dseq)$. Then let  $\Gamma$ be the edge containing $\Gamma^+$, oriented so that $\Gamma^+$ is at the head; for later use, let $\Gamma^-=\Gamma^-(\dseq)$ be the other half-edge of $\Gamma$. We call $(G(\dseq),T(\dseq),\Gamma(\dseq))$, or any other graph with the same distribution, a {\em random tree-rooted graph with degree sequence $\dseq$}. The tree $T$ has now taken office. 

The next proposition describes the distribution of $(G(\dseq),T(\dseq),\Gamma(\dseq))$.  For a tree-rooted graph $(g,t,\gamma)$, 
\begin{prop}\label{eq:gtg_dist}
Let $\dseq=(d(1),\ldots,d(n))$ be a degree sequence with $d(i) \ge 1$ for all $i \in [n]$, and write $m=\frac{1}{2} \sum_{i=1}^n d(i)$. 
Fix a tree-rooted graph $(g,t,\gamma)$ where $g$ is a graph with degree sequence $\dseq$. Then
\[
\p{(G(\dseq),T(\dseq),\Gamma(\dseq))=(g,t,\gamma)} \propto \frac{2^{\I{\gamma~\mathrm{is~a~loop}}}\cdot m_{g-t}(\gamma)}{\prod_{i=1}^n 2^{m_{g-t}(ii)} \cdot \prod_{e \in \edges(g)} m_{g-t}(e)!}\, .
\]
\end{prop}
\begin{proof}
Proposition~\ref{prop:tree_formula} gives us a formula for $\p{T(\dseq)=t}$. We next focus on computing 
\[
\probC{G(\dseq)=g}{T(\dseq)=t}.
\] 
Write $r$ for the root of $t$, and $\gamma=qr$ for the oriented root edge of $g$. 
Given that $T(\dseq)=t$, each $i \in [n]$ with $i \ne r(t)$ has $d'(i):=d(i)-c_t(i)-1$ pendant half-edges attached to it, and $r$ has $d'(r) := d(r)-c_t(r)$ half-edges attached to it. Conditionally given that $T(\dseq)=t$, the graph $G(\dseq)-T(\dseq)$ is distributed as $\mathop{C\!M}(\dseq')$, a random graph with degree sequence $\dseq'=(d'(1),\ldots,d'(n))$ sampled according to the configuration model, so with distribution as in (\ref{eq:cm})), and more 
Writing $m':=m-(n-1)=\frac{1}{2}\sum_{i=1}^n d'(i)$ and $g'=(\vertices(g),\edges(g)\setminus \edges(t))$, it follows that
\begin{align*}
\probC{G(\dseq)=g}{T(\dseq)=t}
& = \p{\mathop{C\!M}(\dseq') = g'} \\
& = 
\frac{2^{m'}(m')!}{(2m')!}\frac{\prod_{i=1}^n d'(i)!}{\prod_{i=1}^n 2^{m_{g'}(ii)} \cdot \prod_{e \in \edges(g')} m_{g'}(e)!}\\
& = \frac{2^{m'}(m')!}{(2m')!} \frac{\prod_{i=1}^n d'(i)!}{\prod_{i=1}^n 2^{m_{g-t}(ii)} \cdot \prod_{e \in \edges(g)} m_{g-t}(e)!}\, .
\end{align*} 
For the second equality we have used the exact expression for the distribution of $\mathop{C\!M}(\dseq')$, which can be found in, e.g., \cite{MR3617364}, equation  (7.2.6). 
For the last equality, we use that $m_{g'}(ii)=m_{g-t}(ii)$ since $t$ is a tree so contains no loops, and that $m_{g'}(e)=m_{g-t}(e)$ by definition when $e \in \edges(g')$. 

Given that $T(\dseq)=t$ and that $G(\dseq)=g$, in order to have $(G(\dseq),T(\dseq),\Gamma(\dseq))=(g,t,\gamma)$ it is necessary and sufficient that $\Gamma(\dseq)=\gamma$. This occurs precisely if $\gamma^+$, the half-edge of $\gamma$ incident to $r$, was matched with some half-edge incident to $q$. Since the matching of half-edges in $G(\dseq)-T(\dseq)$ is chosen uniformly at random, by symmetry the conditional probability that this occurred is $m_{g-t}(\gamma)/d'(r)$ if $\gamma$ is not a loop, and is $2m_{g-t}(\gamma)/d'(r)$ if $\gamma$ is a loop. We may unify these two formulas by writing
\[
\probC{\Gamma(\dseq)=\gamma}{T(\dseq)=t,G(\dseq)=g} = 
\frac{2^{\I{\gamma~\mathrm{is~a~loop}}}m_{g-t}(\gamma)}{d'(r)}.  
\]
Combined with the formula for $\p{T(\dseq)=t}$ from Proposition~\ref{prop:tree_formula},  this gives 
\begin{align*}
& \p{(G(\dseq),T(\dseq),\Gamma(\dseq))=(g,t,\gamma)}\\
& = 
\frac{1}{(2m-n)_{n-1}} \prod_{i=1}^n (d(i)-1)_{c_t(i)} 
\\
& \quad \cdot  
\frac{2^{m'}(m')!}{(2m')!} \frac{\prod_{i=1}^n d'(i)!}{\prod_{i=1}^n 2^{m_{g-t}(ii)} \cdot \prod_{e \in \edges(g)} m_{g-t}(e)!}\\
& \quad \quad \cdot\frac{2^{\I{\gamma~\mathrm{is~a~loop}}}m_{g-t}(\gamma)}{d'(r)}\\
& = \prod_{i=1}^n (d(i)-1)! \cdot
\frac{2^{m-(n-1)}(m-(n-1))!}{2m'(2m-n)!}
\cdot  
\frac{2^{\I{\gamma~\mathrm{is~a~loop}}}m_{g-t}(\gamma)}{\prod_{i=1}^n 2^{m_{g-t}(ii)} \cdot \prod_{e \in \edges(g)} m_{g-t}(e)!}\, .
\end{align*}
In the second equality we have used that $(2m-n)_{n-1}(2m')!=2m'(2m-n)!$, that $(d(i)-1)_{c_t(i)}d'(i)! = (d(i)-1)!$ for $i \ne r$, and that $(d(r)-1)_{c_t(r)}d'(r)! = d'(r) (d(r)-1)!$. The first two terms on the final line do not depend on the triple $(g,t,\gamma)$, so the result follows. 
\end{proof}

\section{Concentration of degrees}
\label{sec:concentration}
Throughout this section, let $(\dseq^n,n \ge 1)$ be a sequence of degree sequences satisfying the conditions of Theorem~\ref{thm:main}, and also let $p^n$ and $p$ be as in Theorem~\ref{thm:main}. Next, for $n \ge 1$ let $T(\dseq^n)$ be the tree built by Pitman's additive coalescent applied to the degree sequence $\dseq^n=(d^n(i),1 \le i \le n)$. Let $\cseq^n=(c^n(i),1 \le i \le n)$ be the child sequence of $T(\dseq^n)$, and recall that  $Q_{\cseq^n}=(Q_{\cseq^n}(a),a \ge 0)$ is the child statistics vector of $\cseq^n$. Also, for $0 \le a < b$, let $P^n_{\gdeg,\tdeg} = \#\{1 \le i \le n: d^n(i)=\gdeg,c^n(i)=\tdeg\}$. Finally, let $\rho := 1/(\mu_1(p)-1)$. Note that 
since $\sum_{i \in [n]} d^n(i) \ge 2n$, necessarily $\mu_1(p^n) \ge 2$; since $p^n \to p$ pointwise and $\mu_2(p^n) \to \mu_2(p)$, it follows that $\mu_1(p^n) \to \mu_1(p)$, so $\mu_1(p) \ge 2$ and hence $\rho \in (0,1]$.
\begin{prop} \label{prop:degree_concentration}
For $\tdeg \ge 0$ let 
\[
q(\tdeg) := \sum_{\gdeg=\tdeg+1}^{\infty} p(\gdeg)\cdot \p{\mathrm{Bin}(\gdeg-1,\binprob)=\tdeg}. 
\]
Then $\mu_2(q) < \infty$ and $\mu_2(n^{-1}Q_{\cseq^n}) \to \mu_2(q)$ in probability as $n \to \infty$. Moreover, for all $0 \le \tdeg < \gdeg$, $n^{-1}P^n_{\gdeg,\tdeg} \convp p(\gdeg)\cdot \p{\mathrm{Bin}(\gdeg-1,\binprob)=\tdeg}$, and $n^{-1}Q_{\cseq^n}(\tdeg) \convp q(\tdeg)$, in both cases as $n \to \infty$. 
\end{prop}
Let $(G(\dseq^n),T(\dseq^n),\Gamma(\dseq^n))$ be a random tree-weighted graph with degree sequence $\dseq^n$. Using Proposition \ref{prop:degree_concentration}, together with existing results from the literature,  it is fairly straightforward to establish that $(\sigma n^{-1/2})T(\dseq^n) \convdist \cT$, with $\sigma=\mu_2(q)-1 \in (0,\infty)$, where $\cT$ is the Brownian continuum random tree. However, in order to show that such convergence holds for the corresponding random {\em simple} tree-weighted graphs, we additionally need the next proposition, which establishes that the number of pairs of tree-adjacent vertices  in $T(\dseq^n)$ with given fixed degrees is well-concentrated around its expected values. This will be used in order to show that the probability of $G(\dseq^n)$ being simple given $T(\dseq^n)$ asymptotically behaves like a constant.

Write $G_-(\dseq^n) = G(\dseq^n)-T(\dseq^n)$ and let $\dseq_-^n=(d^n_-(i),1 \le i \le n)$ be the degree sequence of $G_-(\dseq^n)$. For integers $k,\ell \ge 0$, let 
\begin{equation}\label{eq:alphadef}
\alpha(k,\ell) = 
\sum_{\tdeg_1,\tdeg_2 \ge 0}
\tdeg_2p(\ell+\tdeg_2+1)\p{\mathrm{Bin}(\ell+\tdeg_2,\rho)=\tdeg_2}\cdot p(k+\tdeg_1+1)\p{\mathrm{Bin}(k+\tdeg_1,\rho)=\tdeg_1}\, .
\end{equation} 
\begin{prop}\label{prop:akl_conv}
For integers $k,\ell \ge 0$ let 
\[
A^n(k,\ell) = 
\left|\left\{uv \in \edges(T(\dseq^n)) : \dseq^n_-(u)=k,\dseq^n_-(v)=\ell \right\} \right|.
\] 
Then for all $k,\ell \ge 0$, 
\[
\frac{1}{n}A^n(k,\ell) \convp \alpha(k,\ell)
\]
as $n \to \infty$, and also 
\[
\frac{1}{n}\sum_{k,\ell \ge 0} 
k\ell A^n(k,\ell) \convp \sum_{k,\ell \ge 0} k\ell \alpha(k,\ell). 
\]
\end{prop}
The proofs of Propositions~\ref{prop:degree_concentration} and~\ref{prop:akl_conv} appear in Appendix~\ref{app:1}. 

To conclude the section, we observe that $\alpha(k,\ell)$ defines a probability distribution on pairs of non-negative integers. Indeed, 
\begin{align*}
\sum_{k \ge 0} \sum_{\tdeg_1 \ge 0} 
p(k+\tdeg_1+1)\p{\mathrm{Bin}(k+\tdeg_1,\rho)=\tdeg_1}
& = 
\sum_{m\ge 0} \sum_{a=0}^m
p(m+1)\p{\mathrm{Bin}(m,\rho)=\tdeg}\\
& = \sum_{m \ge 0} p(m+1) = 1-p(0)=1\, ,
\end{align*}
and 
\begin{align*}
& \sum_{\ell \ge 0} 
\sum_{\tdeg_2 \ge 0}
\tdeg_2p(\ell+\tdeg_2+1)\p{\mathrm{Bin}(\ell+\tdeg_2,\rho)=\tdeg_2}\\
& = 
\sum_{m \ge 0} \sum_{a=0}^m
\tdeg p(m+1) \p{\mathrm{Bin}(m,\rho)=\tdeg}\\
& = \sum_{m \ge 0} p(m+1) \cdot m \rho 
=  (\mu_1(p)-(1-p(0)))\rho = (\mu_1(p)-1)\rho = 1,
\end{align*}
so by factorizing $\sum_{k,\ell \ge 0} \alpha(k,\ell)$ we obtain 
\[
\sum_{k,\ell \ge 0} \alpha(k,\ell) 
=
\pran{\sum_{m \ge 0} p(m+1)}
\cdot
\pran{\sum_{m \ge 0}\sum_{m \ge 0}  p(m+1) \cdot m \rho } 
= 1\, ; 
\] 
the fact that $\sum_{k,\ell\ge 0} \alpha(k,\ell)=1$ will be used in the proof of Proposition~\ref{prop:akl_conv}. 
A similar computation shows that 
\begin{equation}\label{eq:klsum_finite} 
\sum_{k,\ell \ge 0} k\ell\cdot\alpha(k, \ell) 
\le 
\pran{\sum_{m \ge 0} mp(m+1)}\cdot 
\pran{\sum_{m \ge 0} m^2 p(m+1) \rho} 
= \mu_2(p)-2\mu_1(p) +1 < \infty\, , 
\end{equation} 
a fact we will use in bounding the probability of simplicity of $G(\dseq^n)$. 

\section{Poisson approximation for graph superpositions.}\label{sec:poisson_approx}
In this section we state a Poisson approximation theorem for the number of loops and multiple edges in the superposition of a fixed simple graph and a random graph with a fixed degree sequence; this in particular allows us to control the probability that such a superposition yields a simple graph. 

Let $H$ be a simple graph with vertex set $v(H)=[n]$. 
Fix a degree sequence $\dseq=(d(1),\ldots,d(n))$ whose sum of degrees is even, and let $G$ be a random graph with degree sequence $\dseq$ sampled according to the configuration model.
For vertices $u,v \in [n]$ and $i \in [d(u)],j \in [d(v)]$, let $\I{ui,vj}$ be the indicator of the event that half-edge $ui$ is matched with half-edge $vj$ in $G$. Now write 
\begin{align*}
\cL& = \cL(G) = \{(ui,uj): u \in [n], i,j \in [d(u)], i<j\}\\
\cM& = \cM(G,H) 
= \{((ui_1,vj_1),(ui_2,vj_2)): u,v \in [n], uv \notin e(H), \\ 
& 
\ \ \ \ \ \ \ \ \ \ \ \ \ \ \ \ \ \ \ \ \ \ \ \ \ \ \
i_1,i_2 \in [d(u)], j_1,j_2 \in [d(v)], u<v, i_1<i_2,j_1 \neq j_2 \}, \text{ and}\\
\cN& = \cN(G,H) = \{(ui,vj): uv \in e(H), i \in [d(u)], j \in [d(v)] \},
\end{align*}  
and let
\begin{align*} 
L 	& = L(G) = \sum_{(ui,uj) \in \cL} \I{ui,uj}  \, ,\\ 
M 	& = M(G,H) = \sum_{((ui_1,vj_1),(ui_2vj_2)) \in \cM} \I{(ui_1,vj_1)} \I{(ui_2vj_2)} \, , \text{ and} \\ 
N	& = N(G,H) \sum_{(ui,vj) \in \cN} \I{ui,vj} \ .
\end{align*} 
Note that the graph with edge set $e(G)\cup e(H)$ is simple precisely if $L+M+N=0$.

\begin{thm} \label{Poisson Asymptotics}
Fix a sequence of simple graphs $(h_n,n \ge 1)$ with $v(h_n)=[n]$ for all $n \geq 1$ and $\max_{v \in [n]} \{\deg_{h_n}(v)\} = o(n)$. For each $n \ge 1$ let $\dseq^{n}=(d^{n}(v),1 \le v \le n)$ be a degree sequence and let $p^n$ be the degree distribution of $\dseq^n$. 
Suppose that there exists a probability distribution $p=(p(k),k \ge 0)$ with $\mu_2(p) \in [0,\infty)$ and $p(0)<1$ such that the following holds.

First, $p^n \to p$ pointwise and $\mu_2(p^n) \to \mu_2(p)$. Second, there are non-negative numbers $(\alpha(a,b),a,b \ge 0)$ such that for any $a,b \ge 0$
\[
\alpha^n(a,b) := \frac{1}{n} \left|\left\{uv \in e(h_n) : d^n(u)=a,d^n(v)=b \right\} \right| \to \alpha(a,b),
\]
and 
\begin{equation} \label{convergence to eta}
\sum_{k,\ell \geq 0} kl\alpha^n(k,\ell) \rightarrow \sum_{k,\ell \geq 0} kl\alpha(k,\ell) < \infty 
\end{equation}

For $n \ge 1$ let $G_n$ be distributed according to the configuration model on graphs with vertex set $[n]$ and degree sequence  $d^n$. 
Then with $L_n=L(G_n)$, $M_n=M(G_n,h_n)$ and $N_n=N(G_n,h_n)$, we have
\[ 
\|\mathrm{Dist}(L_n,M_n,N_n) - \mathrm{Poi}(\nu/2)\otimes 
\mathrm{Poi}(\nu^2/4)\otimes \mathrm{Poi}(\eta)\|_{\mathrm{TV}} \to 0
\]
as $n \to \infty$, where $\nu = \left(\mu_2(p)/\mu_1(p)\right) - 1$ and 
$\eta = \frac{1}{\mu_1(p)} \sum_{i,j \ge 1} ij\alpha(i,j)$. 
\end{thm}
In the statement of Theorem \ref{Poisson Asymptotics} we have introduced the notation $\deg_{h_n}(v)$ for the degree of vertex $v$ in $h_n$, and the notation $\|\mu-\nu\|_{\mathrm{TV}}$ for the total variation distance between probability measures. The proof of Theorem~\ref{Poisson Asymptotics} appears in Appendix~\ref{app:2}. This theorem has the following consequence for random tree-weighted graphs, which we will use in the next section. 
\begin{cor}\label{cor:simple_probability}
Let $(\dseq^n,n \ge 1)$ and $(p^n,n \ge 1)$ be as in  Theorem~\ref{thm:main}, and for $n \ge 1$ let $(G(\dseq^n),T(\dseq^n),\Gamma(\dseq^n))$ be a random tree-weighted graph with degree sequence $\dseq$. Then
\[
\probC{G(\dseq^n)~\mathrm{simple}}{T(\dseq^n)} \convp \exp(-\nu/2-\nu^2/4-\eta)\, ,
\]
 as $n \to \infty$. 
\end{cor}
This corollary follows straightforwardly from Theorem~\ref{Poisson Asymptotics} when $\mu_2(p^n)>2$, in which case $G_-(\dseq^n)=G(\dseq^n)-T(\dseq^n)$ has a linear number of edges. However, when $\mu_2(p^n)=2$, and the graph $G_-(\dseq^n)$ has a sub-linear number of edges a separate argument is needed. The proof of Corollary~\ref{cor:simple_probability} also appears in Appendix~\ref{app:2}.

\section{Proof of Theorem~\ref{thm:main}}
\label{sec:finalproof} 
Let $(\dseq^n,n \ge 1)$ be a sequence of degree sequences satisfying the conditions of Theorem~\ref{thm:main}.  For $n \ge 1$ let $T(\dseq^n)$ be the tree built by Pitman's additive coalescent applied to degree sequence $\dseq^n$, and let $\cseq^n$ be the child sequence of $T(\dseq^n)$. By Proposition~\ref{prop:tree_formula}~(1), conditionally given $\cseq^n$, the tree $T(\dseq^n)$ is uniformly distributed over the set of trees with child sequence $\cseq^n$. 

By Proposition~\ref{prop:degree_concentration}, the child statistics vectors $(Q_{\cseq^n},n \ge 1)$ satisfy that, as $n \to \infty$, for all $a \ge 0$, 
\begin{equation}\label{eq:qconv}
n^{-1}Q_{\cseq^n}(a) \convp q(a), 
\end{equation}
and moreover that $\mu_2(n^{-1}Q_{\cseq^n}(a)) \to \mu_2(q)$. Here $q=(q(a),a \ge 0)$ is as in Proposition~\ref{prop:degree_concentration}, and in particular satisfies $\mu_2(q) <\infty$. We will also need that $\mu_2(q) > 1$, and we now justify this. 

The convergence (\ref{eq:qconv}) and the 
fact that $\mu_2(n^{-1}Q_{\cseq^n}(a)) \to \mu_2(q)$ together imply that $\mu_1(n^{-1}Q_{\cseq^n}(a)) \to \mu_1(q)$. But $\mu_1(n^{-1}Q_{\cseq^n}(a)) = (n-1)/n$ since $Q_{\cseq^n}$ is a child sequence, so necessarily $\mu_1(q)=1$. By the definition of $q$, 
if $\rho = 1$ then $q(1)=p(2)$, and $p(2) < 1$ by  assumption. If $\rho > 1$ then 
\[
q(0) := \sum_{\gdeg=1}^{\infty} p(\gdeg)\cdot \p{\mathrm{Bin}(\gdeg-1,\binprob)=0} > 0, 
 \]
 so again $q(1) \le (1-q(0))< 1$. Thus, we always have $q(1)<1$, which together with the fact that $\mu_1(q)=1$ implies that $\mu_2(q)>1$.

Writing $\sigma=\mu_2(q)-1 \in (0,\infty)$, it then follows by Theorem 1 of~\cite{MR3188597} that 
\[
\overline{T}(\dseq^n):=
\frac{\sigma}{n^{1/2}} T(\dseq^n) \convdist \cT,
\]
in the Gromov-Hausdorff-Prokhorov sense.\footnote{Theorem 1 of~\cite{MR3188597} is stated for plane trees with a fixed degree sequence, rather than labeled trees with a fixed degree sequence. However, as noted by Broutin and Marckert \citep[page 295]{MR3188597}, a straightforward combinatorial argument shows that the same result holds for labeled trees. Also, as stated, the theorem only yields convergence in the Gromov-Hausdorff sense; but the proof proceeds by establishing convergence distributional of coding functions. As explained in \cite[Section 3]{MR3035742}, such proofs immediately yield the stronger Gromov-Hausdorff-Prokhorov convergence.}

We aim to prove the same statement with $\overline{T}(\dseq^n)$ replaced by $\overline{T}_n:=(\sigma/n^{1/2})T_n$, where $(G_n,T_n,\Gamma_n)$ is is a uniformly random simple tree-rooted graph with degree sequence $\dseq^n$. To accomplish this, we use that the law of $(G_n,T_n,\Gamma_n)$ is precisely the conditional law of $(G(\dseq^n),T(\dseq^n),\Gamma(\dseq^n))$ given that $G(\dseq^n)$ is a simple graph.

Writing $\mathbb{K}$ for Gromov-Hausdorff-Prokhorov space as in \cite{MR3035742}, for any bounded continuous function $f:\mathbb{K} \to \R$ we have 
\begin{align*}
\E{f(\overline{T}(\dseq^n))\cdot\I{G(\dseq^n)~\mathrm{simple}}} & =
\e\Big(\mathbf{E}\left(f(\overline{T}(\dseq^n))\cdot\I{G(\dseq^n)~\mathrm{simple}}~\big|~T(\dseq^n)\right)\Big) \\
& = \E{f(\overline{T}(\dseq^n))\cdot 
\probC{G(\dseq^n)~\mathrm{simple}}{T(\dseq^n)}}. 
\end{align*}
Since $\e{f(\overline{T}(\dseq^n))} \to \e{f(\cT)}$, and $\probC{G(\dseq^n)~\mathrm{simple}}{T(\dseq^n)} \convp \exp(-\nu/2-\nu^2/4-\eta)$ by Corollary~\ref{cor:simple_probability}, 
it follows that 
\[
\E{f(\overline{T}(\dseq^n))\cdot\I{G(\dseq^n)~\mathrm{simple}}} \to \exp(-\nu/2-\nu^2/4-\eta) \E{f(\cT)}.
\]
Furthermore, 
\[
\p{G(\dseq^n)~\mathrm{simple}} = \E{\probC{G(\dseq^n)~\mathrm{simple}}{T(\dseq^n)}} \to \exp(-\nu/2-\nu^2/4-\eta),
\]
and therefore 
\[
\mathbf{E}\left(f(\overline{T}(\dseq^n))~\big|~G(\dseq^n)~\mathrm{simple}\right)
= 
\frac{\E{f(\overline{T}(\dseq^n))\cdot\I{G(\dseq^n)~\mathrm{simple}}}}{\p{G(\dseq^n)~\mathrm{simple}}}
\to 
\E{f(\cT)}\, .
\]
Since 
\[
\E{f(\overline{T}_n)} = \E{f(\overline{T}(\dseq^n)~\big|~G(\dseq^n)~\mathrm{simple}}\, ,
\]
the fact that $\overline{T}_n \convdist \cT$ now follows by the Portmanteau theorem. \hfill\qed

\appendix
\section{Proofs of Propositions~\ref{prop:degree_concentration} and~\ref{prop:akl_conv}}\label{app:1}

Before beginning the proofs in earnest, we
state and prove a simple bound on the asymptotic behaviour of maximum degrees and sums of small sets of degrees, for sequences of degree sequences as in Theorems~\ref{thm:main} and~\ref{Poisson Asymptotics}, which will be used multiple times below. 
\begin{fact}\label{fact:maxdeg}
For each $n \ge 1$ let $\dseq^{n}=(d^{n}(v),1 \le v \le n)$ be a degree sequence and let $p^n$ be the degree distribution of $\dseq^n$. 
Suppose that there exists a probability distribution $p=(p(k),k \ge 0)$ such that $p^n \to p$ pointwise and $\mu_2(p^n) \to \mu_2(p) \in [0,\infty)$. Then $\max_{1 \le i \le n} d^n(i) = o(n^{1/2})$. 
Also, for any sets $(A_n,n \ge 1)$ with $A_n \subset [n]$ and $|A_n|=o(n)$, it holds that $\sum_{i \in A_n} d^n(i)=o(n)$.
\end{fact}
\begin{proof}
If $p^n \to p$ pointwise and $\mu_2(p^n) \to \mu_2(p) \in [0,\infty)$, then for all $\eps > 0$ there is $M$ such that 
\[
\liminf_{n \to \infty} \sum_{k=1}^M k^2p^n(k) \ge \mu_2(p) - \eps,
\]
so $\sup_{M \ge 1} \liminf_{n \to \infty} \sum_{k=1}^M k^2p^n(k) \ge \mu_2(p)$. 
If additionally there is $\delta > 0$ such that $\max_{1 \le i \le n} d^n(i) \ge \delta n^{1/2}$ for infinitely many $n$, then 
\[
\limsup_{n \to \infty} 
\mu_2(p^n) \ge 
\delta^2 + \sup_{M \ge 1} \liminf_{n \to \infty} \sum_{k=1}^M k^2p^n(k) > \mu_2(p),
\]
so $\mu_2(p^n) \not \to \mu_2(p)$.

Similarly, for sets $(A_n,n\ge 1)$ as in the statement, since $|A_n|=o(n)$, for any $M \in \N$ we have $\sum_{i \in A_n}^n (d^n(i))^2\I{d^n(i)\le M} = o(n)$, so for any $\eps > 0$ there is $M \in \N$ such that 
\[
\liminf_{n \to \infty} \frac{1}{n} \sum_{i=1}^n (d^n(i))^2\I{d^n(i)\le M}\I{i \not\in A_n} \ge \mu_2(p) - \eps. 
\]
This implies that $\liminf_{n \to \infty} n^{-1} \sum_{i=1}^n (d^n(i))^2\I{i \not\in A_n} \ge \mu_2(p)$. 
If also there is $\delta > 0$ such that $\sum_{i \in A_n} (d^n(i))^2 > \delta n$ for infinitely many $n$, then 
\[
\limsup_{n \to \infty} \mu_2(p^n)=\limsup_{n \to \infty} \frac{1}{n} \sum_{i=1}^n (d^n(i))^2 \ge \mu_1(p)+\delta\, ,
\]
so $\mu_2(p^n) \not\to \mu_2(p)$. 
\end{proof}
Note that the conditions on the degree sequences in both Theorem~\ref{thm:main} and Theorem~\ref{Poisson Asymptotics} allow Fact~\ref{fact:maxdeg} to be applied. 

To prove Proposition~\ref{prop:degree_concentration}, we will  make use of the following lemma, which uses the second moment method to control how subsampling affects degree distributions. The proof of the proposition immediately follows that of the lemma. 
\begin{lem}\label{lem:degree_change_control}
For any integer $\gdeg \ge 1$ there exists $n_0$ such that for all $n \ge n_0$ the following holds. 
Let $\dseq = (d(1),\ldots,d(n))$ be a degree sequence with $d(i) \ge 1$ for all $i\in [n]$ and with  $\sum_{i =1}^n d(i) \ge 2n-1$,  set $S = \bigcup_{i=1}^n \{i1,\ldots,i(d(i)-1)\}$ and write $s=|S|$. Let $\urand$ be a uniformly random subset of $S$ with $|U|=n-1$, and for $1 \le i \le n$ write $U_i = \#\{1 \le j < d(i): (i,j) \in \urand\}$. 
For $0 \le \tdeg < \gdeg$, write $P_{\gdeg,\tdeg} = \#\{1 \le i \le n: (d(i),U_i)=(\gdeg,\tdeg)\}$. 
Then for all $\eps > 0$, 
\[
\p{|P_{\gdeg,\tdeg} - \e P_{\gdeg,\tdeg}| > \eps \e P_{\gdeg,\tdeg}} < \frac{1}{\eps^2} \pran{\frac{1}{\e P_{b,a}}+ \frac{2\gdeg^2}{s}}\, .
\]
\end{lem}
\begin{proof}
We fix $0 \le \tdeg \le \gdeg$ and compute the first and second moments of $P_{\gdeg+1,\tdeg}$; this makes the calculations slightly easier to read than they would be for $P_{\gdeg,\tdeg}$. 

Fix indices $k$ and $\ell$ with $k \neq \ell$ and $d(k)=d(\ell)=\gdeg+1$. 
Since $\urand$ is a uniformly random subset of $S$, by symmetry we have 
\[
\p{|U_k|=\tdeg} = \p{|U_\ell|=\tdeg} = {\gdeg \choose \tdeg} \cdot {s-\gdeg \choose n-1-\tdeg}{s \choose n-1}^{-1},
\]
and 
\[
\p{|U_k|=|U_\ell|=\tdeg} = {\gdeg \choose \tdeg}^2 \cdot {s-2\gdeg \choose n-1-2\tdeg}{s \choose n-1}^{-1}, 
\]
so writing $n_{\gdeg+1} = \#\{1 \le i \le n: d(i) = \gdeg+1\}$, we have 
\begin{equation}\label{eq:mu_ab_formula}
\e{P_{\gdeg+1,\tdeg}} = n_{\gdeg+1}{\gdeg \choose \tdeg} \cdot {s-\gdeg \choose n-1-\tdeg}{s \choose n-1}^{-1}\, 
\end{equation}
and
\begin{align*}
& \V{P_{\gdeg+1,\tdeg}} \\
&
= n_{\gdeg+1}(n_{\gdeg+1}-1) {\gdeg \choose \tdeg}^2 \pran{{s - 2\gdeg \choose n-1-2\tdeg}{s \choose n-1}^{-1}-{s-\gdeg \choose n-1-\tdeg}^2{s \choose n-1}^{-2}}\\
& 
+ n_{b+1}\pran{{b \choose a} {s-b \choose n-1-1}{s \choose n-1}^{-1}
- {b \choose a}^2 {s-b \choose n-1-1}^2{s \choose n-1}^{-2}},
\end{align*}
where the final line accounts for the diagonal terms. 
Bounding the final line from above by $\e P_{b+1,a}$ and cancelling terms in the parenthetical expression in the middle line gives 
\begin{align*}
& \V{P_{\gdeg+1,\tdeg}}-\e P_{b+1,a} \\
& \le  n_{\gdeg+1}(n_{\gdeg+1}-1) {\gdeg \choose \tdeg}^2 
\pran{\frac{(n-1)_{2\tdeg}(s-(n-1))_{2(\gdeg-\tdeg)}}{(s)_{2\gdeg}}- \frac{(n-1)_\tdeg^2(s-(n-1))_{\gdeg-\tdeg}^2}{(s)_\gdeg^2}}. 
\end{align*}
The ratio of the first and the second term in the final parentheses is 
\begin{align*}
\frac{(n-1)_{2\tdeg}}{(n-1)_\tdeg^2} \frac{(s-(n-1))_{2(\gdeg-\tdeg)}}{(s-(n-1))_{\gdeg-\tdeg}^2} \frac{(s)_\gdeg^2}{(s)_{2\gdeg}} &  \le \frac{(s)_\gdeg^2}{(s)_{2\gdeg}} 
 \le \pran{1+\frac{\gdeg}{s-2\gdeg}}^\gdeg \le 1 + \frac{2\gdeg^2}{s}\, ,
\end{align*}
the last bound holding for $\gdeg$ fixed and $s$ large. This gives 
\begin{align*}
\V{P_{\gdeg+1,\tdeg}} & \le 
\e P_{b+1,a} + n_{\gdeg+1}(n_{\gdeg+1}-1) {\gdeg \choose \tdeg}^2 \frac{2\gdeg^2}{s} {s-\gdeg \choose n-1-\tdeg}^2{s \choose n-1}^{-2} \\
& < \e P_{b+1,a} + \frac{2\gdeg^2}{s} (\e P_{\gdeg+1,\tdeg})^2\, ,
\end{align*}
and the lemma follows by Chebyshev's inequality. 
\end{proof}
\begin{proof}[Proof of Proposition~\ref{prop:degree_concentration}]
We first bound $\mu_2(q)$ by writing 
\begin{align*}
\mu_2(q) & = \sum_{\tdeg \ge 0} \tdeg^2 q(\tdeg) \\
		& = \sum_{\tdeg \ge 0} \tdeg^2 \sum_{\gdeg > \tdeg} p(\gdeg)\cdot \p{\mathrm{Bin}(\gdeg-1,\binprob)=\tdeg}\\
		& \le \sum_{\gdeg > 0} \gdeg^2 p(\gdeg)\cdot \sum_{0 \le \tdeg < \gdeg} \p{\mathrm{Bin}(\gdeg-1,\binprob)=\tdeg}\\
		& = \mu_2(p)^2 < \infty \, .
\end{align*}

Next, since $p^n \to p$ pointwise and $\mu_2(p^n) \to \mu_2(p) < \infty$, for any $\eps > 0$ there is $k$ such that 
$\sum_{d \ge k} d^2 p^n(d) < \eps$ and $\sum_{d \ge k} d^2 p(d) < \eps$. 
If node $i$ has $\tdeg$ children in $T(\dseq^n)$ then $d^n(i) \ge a+1$, so it follows that 
\begin{align*}
\sum_{\tdeg = k}^\infty \tdeg^2 \frac{Q_{\cseq^n}(\tdeg)}{n}
& = \sum_{\tdeg \ge k} \sum_{\gdeg >\tdeg} \tdeg^2 \frac{\#\{i \le n: c^n(i)=\tdeg,d^n(i)=\gdeg\}}{n} \\
& \le \sum_{\gdeg > k} \gdeg^2 \sum_{\tdeg < \gdeg}\frac{\#\{i \le n: c^n(i)=\tdeg,d^n(i)=\gdeg\}}{n} \\
& = \sum_{\gdeg > k} \gdeg^2 p^n(\gdeg) < \eps\, . 
\end{align*}
To complete the proof it thus suffices to show that 
$n^{-1}P^n_{\gdeg,\tdeg} \to p(\gdeg)\cdot \p{\mathrm{Bin}(\gdeg-1,\binprob)=\tdeg}$ in probability for all $0 \le \tdeg < \gdeg$ and that $n^{-1}Q_{\cseq^n}\to q$ pointwise in probability; the fact that $\mu_2(n^{-1}Q_{\cseq^n}) \to \mu_2(q)$ in probability then immediately follows. 

By the third statement of Proposition~\ref{prop:tree_formula}, the set of non-root half-edges in $T(\dseq^n)$ is a uniformly random size-$(n-1)$ subset of the set $S^n:= \bigcup_{1 \le i \le n} \{i1,\ldots,i(d^n(i)-1)\}$. We will apply Lemma~\ref{lem:degree_change_control} to control the numbers of nodes with a given number of children in $T(\dseq^n)$. To make the coming applications of that lemma transparent, we write $s^n := |S^n| = \sum_{1 \le i \le n} (d^n(i)-1)$.

We handle the cases $\mu_1(p)=2$ and $\mu_1(p) > 2$ separately. If $\mu_1(p)=2$ then $|S^n| = \sum_{1 \le i \le n} (d^n(i)-1) = (1+o(1)) n$ as $n \to \infty$. Note that in this case $\rho(p)=1/(\mu_1(p)-1)=1$ so $q(\tdeg)=p(\tdeg+1)$ for all $\tdeg \ge 0$. 
For any $\tdeg \ge 0$, by (\ref{eq:mu_ab_formula}) we then have 
\[
\e P^n_{\tdeg+1,\tdeg} = (1-o(1)) np^n(\tdeg+1) {\tdeg \choose \tdeg} { |S^n|-1-\tdeg \choose n-1 -\tdeg} {|S_n|\choose n-1}^{-1} = (1-o(1)) np^n(\tdeg+1).
\]
If $p(a+1) > 0$ then $n p^n(a+1) = \Theta(n)$, so
\[
\frac{1}{\E{P^n_{a+1,a}}} + \frac{2(a+1)^2}{|S^n|} = o(1),
\] 
and hence by Lemma~\ref{lem:degree_change_control}, 
\[
\frac{P^n_{\tdeg+1,\tdeg}}{n} \convp p(\tdeg+1)= q(\tdeg). 
\]
If $p(a+1)=0$ then $p^n(a+1) = o(1)$, so $\E{P^n_{a+1,a}}/n \to 0$ and thus $P^n_{\tdeg+1,\tdeg}/n \convp 0=q(\tdeg)$ by Markov's inequality. Since this holds for all $a \ge 0$, and $\sum_{\tdeg \ge 0} P^n_{\tdeg+1,\tdeg}/n \le 1=\sum_{a \ge 0} q(a)$, it follows that $\sum_{\tdeg \ge 0} P^n_{\tdeg+1,\tdeg}/n \to 1$ in probability. 
This implies that $\sum_{\gdeg > \tdeg+1} P^n_{\gdeg,\tdeg}/n \to 0$ in probability, 
so 
\[
\frac{Q_{\cseq^n}(\tdeg)}{n} = \frac{1}{n} \sum_{\gdeg > \tdeg} P^n_{\gdeg,\tdeg} = \frac{P^n_{\tdeg+1,\tdeg}}{n} + \sum_{\gdeg > \tdeg+1} \frac{P^n_{\gdeg,\tdeg}}{n} \convp q(\tdeg)\, ,
\]
and that for all $\gdeg > \tdeg+1$, 
$P^n_{\gdeg,\tdeg}/n \to 0=p(b)\cdot\p{\mathrm{Bin}(b-1,\binprob)=a}$ in probability, as required.

We now assume $\mu_1(p)>2$, so that $\rho(p)=1/(\mu_1(p)-1)<1$. 
Since $p=(p_k,k \ge 1)$ is supported on the positive integers, 
\[
\sum_{\tdeg \ge 0} q(\tdeg) = \sum_{\tdeg \ge 0} \sum_{\gdeg=\tdeg+1}^\infty p(\gdeg)\p{\mathrm{Bin}(\gdeg-1,\binprob)=\tdeg}
 = 
\sum_{\gdeg \ge 1} p(\gdeg)=1\, . 
\]
Recalling that $Q_{\cseq^n}(\tdeg) = \sum_{\gdeg > \tdeg} P^n_{\gdeg,\tdeg}$, 
to show that $n^{-1}Q_{\cseq^n}(\tdeg) \to q(\tdeg)$ in probability, it therefore suffices to prove that 
$P^n_{\gdeg+1,\tdeg}/n \to p(\gdeg+1)\cdot \p{\mathrm{Bin}(\gdeg,\binprob)=\tdeg}$ for all $0 \le \tdeg \le \gdeg$, and we now turn to this. 

Since $\mu_1(p^n) \to \mu_1(p)$, it follows that 
$|\sum_{i=1}^n d^n(i) - \mu_1(p) n| = n|\mu_1(p^n)-\mu_1(p)| = o(n)$ as $n \to \infty$, so $s^n = (1+o(1))n(\mu_1(p)-1)$. Thus, for any $\gdeg \ge 1$ and $0 \le \tdeg \le \gdeg$ we have 
\begin{align*}
{\gdeg \choose \tdeg} \cdot {s^n-\gdeg \choose n-1-\tdeg}{s^n \choose n-1}^{-1}
& = {\gdeg \choose \tdeg}\frac{(n-1)_\tdeg(s^n-(n-1))_{\gdeg-\tdeg}}{(s^n)_\gdeg}\\
& = (1-o(1)){\gdeg \choose \tdeg} \frac{n^\tdeg ((\mu_1(p)-2)n)^{\gdeg-\tdeg}}{((\mu_1(p)-1) n)^\gdeg} \\
& = (1-o(1)){\gdeg \choose \tdeg} \frac{(\mu_1(p)-2)^{\gdeg-\tdeg}}{(\mu_1(p)-1)^\gdeg} \\
& = (1-o(1))\p{\mathrm{Bin}(\gdeg,\binprob)=\tdeg}. 
\end{align*}
Using (\ref{eq:mu_ab_formula}) we thus have 
\[
\e P^n_{\gdeg+1,\tdeg} = (1-o(1)) np^n(\gdeg+1) 
\p{\mathrm{Bin}(\gdeg,\binprob)=\tdeg}
=(1-o(1)) np(\gdeg+1) 
\p{\mathrm{Bin}(\gdeg,\binprob)=\tdeg}\, ,
\]
so again, applying Lemma~\ref{lem:degree_change_control} in the case that $p(b+1)>0$, and applying Markov's inequality in the case that $p(b+1)=0$, we obtain that, as $n \to \infty$, 
\[
\frac{P^n_{\gdeg+1,\tdeg}}{n} \convp p(\gdeg+1)\p{\mathrm{Bin}(\gdeg,\binprob)=\tdeg}\, ,
\]
as required.
\end{proof}
We now turn to controlling the joint degrees of pairs of tree-adjacent vertices in tree-weighted graphs. 
Given a degree sequence $\dseq=(d(1),\ldots,d(n))$ and a tree $t$ with $\vertices(t)=[n]$, 
for integers $\gdeg_1,\gdeg_2,\tdeg_1,\tdeg_2$ let 
\begin{align*}
& R_{\gdeg_1,\gdeg_2,\tdeg_1,\tdeg_2}(t,\dseq) \\
& = \#\{u \in \vertices(t)\setminus\{r(t)\}: d(u)=\gdeg_1, d(\parent(u))=\gdeg_2, c_{t}(u)=\tdeg_1, c_{t}(\parent(u))=\tdeg_2\}\\
& = \sum_{u \in \vertices(t)\setminus \{r(t)\}} 
\I{d(u)=\gdeg_1,c_{t}(u)=\tdeg_1}\cdot 
\I{d(\parent(u))=\gdeg_2,c_{t}(\parent(u))=\tdeg_2}\, .
\end{align*}
If $(g,t,\gamma)$ is a tree-rooted graph and $g$ has degree sequence $\dseq$, then  
$R_{\gdeg_1,\gdeg_2,\tdeg_1,\tdeg_2}(t,\dseq)$ counts the number of edges $xy$ of~$t$ with $y=\parent(x)$ such that $c_t(x)=\tdeg_1$, $c_t(y)=\tdeg_2$ and $d_g(x)=\gdeg_1$, $d_g(y)=\gdeg_2$. 
\begin{prop} 
\label{prop:r_asymptotic}
Under the assumptions of Theorem~\ref{thm:main}, 
for any integers $0 \le \tdeg_1 < \gdeg_1$ and $0 \le \tdeg_2 < \gdeg_2$, as $n \to \infty$, 
\[
\frac{R_{\gdeg_1,\gdeg_2,\tdeg_1,\tdeg_2}(T(\dseq^n),\dseq^n)}{n} 
\to \tdeg_2 p(\gdeg_2)\p{\mathrm{Bin}(\gdeg_2-1,\binprob)=\tdeg_2}
\cdot p(\gdeg_1)\p{\mathrm{Bin}(\gdeg_1-1,\binprob)=\tdeg_1}
\]
in probability, where $\rho=1/(\mu_1(p)-1)$. \end{prop}
We introduce two pieces of notation before beginning the proof. For a half-edge $h$ we write $v(h)$ for the vertex incident to $h$. Also, for $r \in \R$ we write $r_+ := \max(r,0)$. 
\begin{proof}
First, if $\tdeg_2=0$ then the right-hand side is zero, and also $R_{\gdeg_1,\gdeg_2,\tdeg_1,\tdeg_2}(T(\dseq^n),\dseq^n)=0$, since if  $v=\parent(u) \in T(\dseq^n)$ then $c_{T(\dseq^n)}(v) \ge 1$. The result thus holds trivially when $\tdeg_2=0$, and we assume hereafter that $\tdeg_2 \ge 1$. For the remainder of the proof we write $R_{\gdeg_1,\gdeg_2,\tdeg_1,\tdeg_2}=R_{\gdeg_1,\gdeg_2,\tdeg_1,\tdeg_2}(T(\dseq^n),\dseq^n)$ for succinctness. 

Let $\cH$ be a fixed, size-$(n-1)$ subset of $S^n := \bigcup_{1 \le i \le n} \{i1,\ldots,i(d^n(i)-1)\}$. 
Write $\cS^n=\{s_1,\ldots,s_{n-1}\}$ for the (unordered) set of non-root half-edges of $T(\dseq^n)$.
We now show that for any half edge $h \in \cH$ and any root half-edge $r$ with $v(r) \neq v(h)$, for all $1 \le i \le n-1$, 
\begin{equation}\label{eq:exchangeability_1}
\p{(r_i,s_i)=(r,h)~|~\cS^n=\cH} = \p{(r_1,s_1)=(r,h)~|~\cS^n=\cH} \, .
\end{equation}
To see this, note that by the second assertion of Proposition~\ref{prop:tree_formula}, the number of execution paths with $\cS^n=\cH$ is $((n-1)!)^2$. We claim that for any $i \in [n-1]$, the number of execution paths with $\cS^n=\cH$ which additionally satisfy that $(r_i,s_i)=(r,h)$ is $((n-2)!)^2$. As this number does not depend on $i \in [n-1]$, the displayed identity follows from this claim. 

To prove the claim, simply note that there are $(n-2)!$ possible orderings of $\cH$ consistent with the constraint that $s_i=h$. Having fixed such an ordering $(h_1,\ldots,h_{n-1})$, for each $j \in [n-1]$ with $j \ne i$, if $s_k=h_k$ for $1 \le k < j$ then, excluding $r_i$ there are $n-j-\I{j<i}$ unpaired root half-edges in components different from that of $s_j$, and any such root half-edge may be chosen as $r_j$. Thus the number of execution paths with $\cS^n=\cH$ and such that $(r_i,s_i)=(r,h)$ is $(n-2)! \cdot \prod_{j \in [n-1]\setminus \{i\}}(n-j-\I{j < i}) = ((n-2)!)^2$.

Now fix a second non-root half-edge $h' \ne h$ and a second root half-edge $r'\ne r$ not incident to the same vertex as $h'$. 
Then a similar argument to the one leading to (\ref{eq:exchangeability_1}) shows that that for any $1 \le i < j \le n$, 
\begin{align}\label{eq:exchangeability_2}
& \p{(r_i,s_i)=(r,h),(r_j,s_j)=(r',s')~|~\cS^n=\cH} \nonumber\\
& = \p{(r_1,s_1)=(r,h),(r_2,s_2)=(r',s')~|~\cS^n=\cH} \, .
\end{align}
In the current case, the number of execution paths leading to the events in both the left- and right-hand probabilities is $((n-3)!)^2$. 

We will next use the above identities in order to perform first and second moment computations. 
For any set $H \subset S^n$, for $0 \le \tdeg < \gdeg$
let $V^n_{\gdeg,\tdeg}(H) = \{i \in [n]: d^n(i)=\gdeg, |H \cap \{i1,\ldots,i(d^n(i)-1)|=\tdeg\}$. Note that $V^n_{\gdeg,\tdeg}(\cS^n)$ is simply the set of nodes with degree $\gdeg$ in $G(\dseq^n)$ and with $a$ children in $T(\dseq^n)$; so  $P^n_{\gdeg,\tdeg}=|V^n_{\gdeg,\tdeg}(\cS^n)|$. 

Fix a non-root node $u \in T(\dseq^n)$, and let $m\in [n-1]$ be such that $e_m = \{\parent(u),u\}$. Then $v(r_m)=u$ and $v(s_m)=\parent(u)$, so $u \in R_{\gdeg_1,\gdeg_2,\tdeg_1,\tdeg_2}$ if and only if $v(r_m) \in V^n_{\gdeg_1,\tdeg_1}(\cS^n)$ and $v(s_m)\in V^n_{\gdeg_2,\tdeg_2}(\cS^n)$. 
By (\ref{eq:exchangeability_1}), it follows that 
\begin{align} 
& \E{R_{\gdeg_1,\gdeg_2,\tdeg_1,\tdeg_2}~|~\cS^n=\cH} \nonumber\\
& = (n-1) \p{v(r_1) \in V^n_{\gdeg_1,\tdeg_1}(\cS^n), v(s_1) \in V^n_{\gdeg_2,\tdeg_2}(\cS^n)~|~\cS^n=\cH}.
\label{eq:firstmoment_ident}
\end{align}
Likewise, by (\ref{eq:exchangeability_2}) it follows that
\begin{align} 
& \E{{R_{\gdeg_1,\gdeg_2,\tdeg_1,\tdeg_2} \choose 2}~|~\cS^n=\cH} \nonumber\\
& = {n-1 \choose 2} \p{v(r_1),v(r_2) \in V^n_{\gdeg_1,\tdeg_1}(\cS^n),v(s_1),v(s_2) \in V^n_{\gdeg_2,\tdeg_2}(\cS^n)~|~\cS^n=\cH}.\label{eq:secondmoment_ident}
\end{align}
We develop the latter two identities in turn.

For integers $0 \le \tdeg < \gdeg$, the number of non-root half-edges $h \in \cS^n$ with $v(h) \in V^n_{\gdeg,\tdeg}(\cS^n)$ is $\tdeg \cdot |V^n_{\gdeg,\tdeg}(\cS^n)|$, and the number of root half-edges $h$ with $v(h)\in V^n_{\gdeg,\tdeg}(\cS^n)$ is just $|V^n_{\gdeg,\tdeg}(\cS^n)|$. 
Conditionally given that $\cS^n=\cH$, the half-edge $s_1$ is a uniformly random element of $\cH$, so
\[
\p{v(s_1) \in V^n_{\gdeg_2,\tdeg_2}(\cS^n)~|~\cS^n=\cH} = \frac{\tdeg_2|V^n_{\gdeg_2,\tdeg_2}(\cH)|}{|\cH|} = \frac{\tdeg_2|V^n_{\gdeg_2,\tdeg_2}(\cH)|}{n-1}\, .
\]
Having chosen $s_1$, if $v(s_1)=v$ then $v(r_1)$ is a uniformly random element of $[n]\setminus \{v\}$, so 
\[
\p{ v(r_1) \in V^n_{\gdeg_1,\tdeg_1}(\cS^n)~|~\cS^n=\cH,v(s_1) \in V^n_{\gdeg_2,\tdeg_2}(\cS^n)}
= \frac{(|V^n_{\gdeg_1,\tdeg_1}(\cH)|-\I{(\gdeg_1,\tdeg_1)\!=\!(\gdeg_2,\tdeg_2)})_+}{n-1}. 
\]
Using these two identities in  (\ref{eq:firstmoment_ident}), it follows that 
\[
(n-1)\E{R_{\gdeg_1,\gdeg_2,\tdeg_1,\tdeg_2}~|~\cS^n=\cH}
= \tdeg_2|V^n_{\gdeg_2,\tdeg_2}(\cH)|(|V^n_{\gdeg_1,\tdeg_1}(\cH)|-\I{(\gdeg_1,\tdeg_1)\!=\!(\gdeg_2,\tdeg_2)})_+\, ,
\]
so since $|V^n_{\gdeg,\tdeg}(\cS^n)|=P^n_{\gdeg,\tdeg}$ for all $0 \le \tdeg < \gdeg$, by Proposition~\ref{prop:degree_concentration} we have 
\begin{align}
\E{\frac{R_{\gdeg_1,\gdeg_2,\tdeg_1,\tdeg_2}}{n}~|~\cS^n} 
&= \frac{1}{n(n-1)}\tdeg_2 P^n_{\gdeg_2,\tdeg_2}(P^n_{\gdeg_1,\tdeg_1}-\I{(\gdeg_1,\tdeg_1)\!=\!(\gdeg_2,\tdeg_2)})
\nonumber\\
& \convp  \tdeg_2 p(\gdeg_2)\p{\mathrm{Bin}(\gdeg_2-1,\binprob)=\tdeg_2}
\cdot p(\gdeg_1)\p{\mathrm{Bin}(\gdeg_1-1,\binprob)=\tdeg_1}\, .\label{eq:firstmoment_conditional}
\end{align}

For the second moment calculation, we need to additionally compute 
\begin{align} \label{eq:big ol' conditional thing}
\p{v(r_2) \in V^n_{\gdeg_1,\tdeg_1}(\cS^n),v(s_2) \in V^n_{\gdeg_2,\tdeg_2}(\cS^n)~|~\cS^n=\cH,
v(r_1) \in V^n_{\gdeg_1,\tdeg_1}(\cS^n), v(s_1) \in V^n_{\gdeg_2,\tdeg_2}(\cS^n)}\, .
\end{align}
Under the conditioning in (\ref{eq:big ol' conditional thing}), the number of non-root half-edges $h \in \cS^n\setminus \{s_1\}$ with $v(h) \in V^n_{\gdeg_2,\tdeg_2}(\cH)$ is $(\tdeg_2\cdot |V^n_{\gdeg_2,\tdeg_2}(\cH)|-1)_+$, so 
\begin{align*}
& \p{v(s_2) \in V^n_{\gdeg_2,\tdeg_2}(\cS^n)~|~\cS^n=\cH,
v(r_1) \in V^n_{\gdeg_1,\tdeg_1}(\cS^n), v(s_1) \in V^n_{\gdeg_2,\tdeg_2}(\cS^n)
}\\
& = \frac{(\tdeg_2\cdot |V^n_{\gdeg_2,\tdeg_2}(\cH)|-1)_+}{n-2}\, .
\end{align*}
Now suppose 
that $\cS^n=\cH,
v(r_1) \in V^n_{\gdeg_1,\tdeg_1}(\cS^n), v(s_1) \in V^n_{\gdeg_2,\tdeg_2}(\cS^n)$, and 
that $v(s_2) \in V^n_{\gdeg_2,\tdeg_2}(\cH)$, and consider the number of possible values for $r_2$. We claim that the number of unpaired root half-edges $h$ with $v(h) \in V^n_{\gdeg_1,\tdeg_1}(\cS^n)$ such that $v(h)$ is in a component different from $v(s_2)$ is 
\[
(|V^n_{\gdeg_1,\tdeg_1}(\cH)|-1-\I{(\gdeg_1,\tdeg_1)\!=\!(\gdeg_2,\tdeg_2)})_+. 
\] 
To see this, note that if $(\gdeg_1,\tdeg_1)=(\gdeg_2,\tdeg_2)$ and 
either $v(s_2)=v(s_1)$ or $v(s_2)=v(r_1)$, then we are precisely constrained constrained to choose $h$ so that $v(h) \in V^n_{\gdeg_1,\tdeg_1}(\cH)\setminus \{v(r_1),v(s_1)\}$. On the other hand, if $(\gdeg_1,\tdeg_1)=(\gdeg_2,\tdeg_2)$ and $v(s_2) \not\in \{v(r_1),v(s_1)\}$ then we are constrained to choose $h$ so that $v(h) \in V^n_{\gdeg_1,\tdeg_1}(\cH)\setminus \{v(r_1),v(s_2)\}$. Both cases agree with the above formula. When  $(\gdeg_1,\tdeg_1)=(\gdeg_2,\tdeg_2)$, the claim is straightforward, since in that case we are only constrained to choose $h$ so that $v(h) \in V^n_{\gdeg_1,\tdeg_1}(\cH)\setminus\{v(r_1)\}$. 
It follows that 
\begin{align*}
& \p{v(r_2) \in V^n_{\gdeg_1,\tdeg_1}(\cS^n)~|~\cS^n=\cH,
v(s_2) \in V^n_{\gdeg_2,\tdeg_2}(\cS^n), v(r_1) \in V^n_{\gdeg_1,\tdeg_1}(\cS^n), v(s_1) \in V^n_{\gdeg_2,\tdeg_2}(\cS^n)
}\\
& = 
\frac{(|V^n_{\gdeg_1,\tdeg_1}(\cH)|-1-\I{(\gdeg_1,\tdeg_1)\!=\!(\gdeg_2,\tdeg_2)})_+}{n-2}\, .
\end{align*}
Combining the above identities with (\ref{eq:secondmoment_ident}) yields that 
\begin{align*}
& 2(n-1)(n-2) \E{{R_{\gdeg_1,\gdeg_2,\tdeg_1,\tdeg_2} \choose 2}~|~\cS^n=\cH} \\
& = \tdeg_2|V^n_{\gdeg_2,\tdeg_2}(\cH)|\big(\tdeg_2|V^n_{\gdeg_2,\tdeg_2}(\cH)|-1\big)_+
\\
& \cdot 
\big(|V^n_{\gdeg_1,\tdeg_1}(\cH)|-\I{(\gdeg_1,\tdeg_1)\!=\!(\gdeg_2,\tdeg_2)}\big)_+
\big(|V^n_{\gdeg_1,\tdeg_1}(\cH)|-1-\I{(\gdeg_1,\tdeg_1)\!=\!(\gdeg_2,\tdeg_2)}\big)_+,
\end{align*}
so since $\tdeg_2 \ge 1$, Proposition~\ref{prop:degree_concentration} implies that
\begin{align*}
& \E{\frac{R_{\gdeg_1,\gdeg_2,\tdeg_1,\tdeg_2}(R_{\gdeg_1,\gdeg_2,\tdeg_1,\tdeg_2}-1)}{n^2}~|~\cS^n}\\
& = \frac{\tdeg_2 P^n_{\gdeg_2,\tdeg_2}(\tdeg_2 P^n_{\gdeg_2,\tdeg_2}-1\big)_+}{n(n-1)} 
\cdot \frac{(P^n_{\gdeg_1,\tdeg_1}-\I{(\gdeg_1,\tdeg_1)\!=\!(\gdeg_2,\tdeg_2)}\big)_+
(P^n_{\gdeg_1,\tdeg_1}-1-\I{(\gdeg_1,\tdeg_1)\!=\!(\gdeg_2,\tdeg_2)}\big)_+}{n(n-2)}\\
& 
\convp  \big(\tdeg_2 p(\gdeg_2)\p{\mathrm{Bin}(\gdeg_2-1,\binprob)=\tdeg_2}
\cdot p(\gdeg_1)\p{\mathrm{Bin}(\gdeg_1-1,\binprob)=\tdeg_1}\big)^2\, .
\end{align*}
Also, (\ref{eq:firstmoment_conditional}) implies that $\E{n^{-2}R_{\gdeg_1,\gdeg_2,\tdeg_1,\tdeg_2}~|~\cS^n} \convp 0$, which with the preceding asymptotic implies that 
\[
\E{\frac{R_{\gdeg_1,\gdeg_2,\tdeg_1,\tdeg_2}^2}{n^2}~|~\cS^n} \convp
\big(\tdeg_2 p(\gdeg_2)\p{\mathrm{Bin}(\gdeg_2-1,\binprob)=\tdeg_2}
\cdot p(\gdeg_1)\p{\mathrm{Bin}(\gdeg_1-1,\binprob)=\tdeg_1}\big)^2\, .
\]
Combining this with (\ref{eq:firstmoment_conditional}) gives that 
\[
\E{\pran{\frac{R_{\gdeg_1,\gdeg_2,\tdeg_1,\tdeg_2}}{n}}^2~|~\cS^n}
-
\pran{\E{\frac{R_{\gdeg_1,\gdeg_2,\tdeg_1,\tdeg_2}}{n}~|~\cS^n}}^2 \convp 0\, ;
\]
the conditional Chebyshev's inequality then gives that for all $\eps > 0$, 
\[
\p{\left|\frac{R_{\gdeg_1,\gdeg_2,\tdeg_1,\tdeg_2}}{n}-\E{\frac{R_{\gdeg_1,\gdeg_2,\tdeg_1,\tdeg_2}}{n}~|~\cS^n}\right|>\eps~|~\cS^n} \convp 0. 
\]
Taking expectations on the left of the previous inequality to remove the conditioning, and again using (\ref{eq:firstmoment_conditional}), this time to replace the term $\E{\frac{R_{\gdeg_1,\gdeg_2,\tdeg_1,\tdeg_2}}{n}~|~\cS^n}$ in the probability by the constant $C:=\tdeg_2 p(\gdeg_2)\p{\mathrm{Bin}(\gdeg_2-1,\binprob)=\tdeg_2}
\cdot p(\gdeg_1)\p{\mathrm{Bin}(\gdeg_1-1,\binprob)=\tdeg_1}$, we obtain that 
\[
\p{\left|\frac{R_{\gdeg_1,\gdeg_2,\tdeg_1,\tdeg_2}}{n}-C\right|>\eps} \to 0,
\]
as required. 
\end{proof}
\begin{proof}[Proof of Proposition~\ref{prop:akl_conv}]
We may reexpress $A^n(k,\ell)$ as
\begin{align*}
A^n(k,\ell) 
& = 
\sum_{\tdeg_1,\tdeg_2 \ge 0} 
\sum_{u\in \vertices(T(\dseq^n))} \I{r(T(\dseq^n)) \not\in \{u,\parent(u)\}}\\
& \qquad\qquad\qquad\qquad\cdot
\I{\dseq^n(u)=k+\tdeg_1+1,\cseq^n(u)=\tdeg_1}\\
& \qquad\qquad\qquad\qquad\cdot 
\I{\dseq^n(\parent(u))=\ell+\tdeg_2+1,
\cseq^n(\parent(u))=\tdeg_2}\\
& + 
\sum_{\tdeg_1,\tdeg_2 \ge 0} 
\sum_{u\in \vertices(T(\dseq^n))} 
\I{\parent(u)=r(T(\dseq^n))}\\
& \qquad\qquad\qquad\qquad\cdot
\I{\dseq^n(u)=k+\tdeg_1 + 1,\cseq^n(u)=\tdeg_1}\\
& \qquad\qquad\qquad\qquad\cdot 
\I{\dseq^n(\parent(u))=\ell+\tdeg_2,
\cseq^n(\parent(u))=\tdeg_2}
\end{align*}
For fixed $\tdeg_1,\tdeg_2 \ge 0$, 
if we replace $\I{r(T(\dseq^n)) \not\in \{u,\parent(u)\}}$ by $\I{u \ne r(T(\dseq^n))}$ in the first double sum, then the inner sum is simply $R_{k+\tdeg_1+1,\ell+\tdeg_2+1,\tdeg_1,\tdeg_2}$. It follows that 
\begin{align*}
A^n(k,\ell) 
& = 
\sum_{\tdeg_1,\tdeg_2 \ge 0} 
R_{k+\tdeg_1+1,\ell+\tdeg_2+1,\tdeg_1,\tdeg_2} \\
& + 
\sum_{\tdeg_1,\tdeg_2 \ge 0} 
\sum_{u\in \vertices(T(\dseq^n))} 
\I{\parent(u)=r(T(\dseq^n))}\\
& \qquad\qquad\qquad\qquad\cdot
\I{\dseq^n(u)=k+\tdeg_1 + 1,\cseq^n(u)=\tdeg_1}\\
& \qquad\qquad\qquad\qquad\cdot 
\I{\dseq^n(\parent(u))=\ell+\tdeg_2,
\cseq^n(\parent(u))=\tdeg_2}
\\
& - 
\sum_{\tdeg_1,\tdeg_2 \ge 0} 
\sum_{u\in \vertices(T(\dseq^n))} 
\I{\parent(u)=r(T(\dseq^n))}\\
& \qquad\qquad\qquad\qquad\cdot
\I{\dseq^n(u)=k+\tdeg_1+1,\cseq^n(u)=\tdeg_1}\\
& \qquad\qquad\qquad\qquad\cdot 
\I{\dseq^n(\parent(u))=\ell+\tdeg_2+1,
\cseq^n(\parent(u))=\tdeg_2}\, .
\end{align*}
But each of the last two double sums is bounded by $c^n(r(T(\dseq^n)))$, since they both count each child of the root at most once. Under the assumptions of Theorem~\ref{thm:main}, by Fact~\ref{fact:maxdeg} we have $c^n(r(T(\dseq^n))) \le \max_{1 \le i \le n} d^n(i) = o(n^{1/2})$,
 so the preceding identity gives
\[
\left|A^n(k,\ell) 
 -
\sum_{\tdeg_1,\tdeg_2 \ge 0} 
R_{k+\tdeg_1+1,\ell+\tdeg_2+1,\tdeg_1,\tdeg_2}\right|=o(n^{1/2}). 
\]
Since 
\begin{align*}
&n^{-1}
R_{k+\tdeg_1+1,\ell+\tdeg_2+1,\tdeg_1,\tdeg_2}\\
 \\
& \convp\tdeg_2p(\ell+\tdeg_2+1)\p{\mathrm{Bin}(\ell+\tdeg_2,\rho)=\tdeg_2}\cdot p(k+\tdeg_1+1)\p{\mathrm{Bin}(k+\tdeg_1,\rho)=\tdeg_1}
 \end{align*}
by Proposition~\ref{prop:r_asymptotic}, and summing the right-hand side of the last expression over $a_1,a_2 \ge 0$ gives $\alpha(k,l)$, it follows that for any $\eps > 0$, 
\[
\p{A^n(k,l)/n \ge \alpha(k,l)-\eps} \to 1. 
\]
But also $n^{-1} \sum_{k,l \ge 0} A^n(k,l) = |\edges(T(\dseq^n))| = (n-1)/n \to 1$; so since $\sum_{k,l\ge 0} \alpha(k,l)=1$, we must in fact have that
\[
\frac{A^n(k,l)}{n} \convp \alpha(k,l)
\]
for all $k,l \ge 0$, as required. 

It remains to show that $n^{-1}\sum_{k,l \ge 0} kl A^n(k,l) \convp \sum_{k,l\ge 0} kl\alpha(k,l)$. 
For this we will exploit the exchangeability of Pitman's additive coalescent. Recall the notation $v(h)$ for the vertex incident to half-edge $h$. 
Note that for any $M \in \N$ we have 
\begin{align*}
\sum_{k,l\ge 0} klA^n(k,l) - \sum_{0 \le k,l\le M} klA^n(k,l) 
& = \sum_{uv \in e(T^n)} d^n_-(u)d^n_-(v) \I{\max(d^n_-(u),d^n_-(v)) > M} \\
& 
\le \sum_{uv \in e(T^n)} d^n(u)d^n(v) \I{\max(d^n(u),d^n(v)) > M} \\
& = 
\sum_{i=1}^{n-1} d^n(v(r_i))d^n(v(s_i)) \I{\max(d^n(v(r_i)),d^n(v(s_i))) > M}. 
\end{align*}
Now, by Proposition~\ref{prop:tree_formula}~(3) and the identity (\ref{eq:exchangeability_1}),  for any $1 \le i \le n-1$ we have 
\begin{align*}
& \E{d^n(v(r_i))d^n(v(s_i)) \I{\max(d^n(v(r_i)),d^n(v(s_i))) > M}} \\
& = \E{d^n(v(r_1))d^n(v(s_1)) \I{\max(d^n(v(r_1)),d^n(v(s_1))) > M}},
\end{align*}
and by the definition of Pitman's additive coalescent we have 
\begin{align*}
& \E{d^n(v(r_1))d^n(v(s_1))\I{\max(d^n(v(r_1)),d^n(v(s_1))) > M}}
\\
& = 
\sum_{u \in [n]}\sum_{v \in [n]}
d^n(u)d^n(v)\I{\max(d^n(u),d^n(v))) > M}\p{v(s_1)=u,v(r_1)=v}\\
& = 
\sum_{u \in [n]}\sum_{v \in [n]}
d^n(u)d^n(v)\I{\max(d^n(u),d^n(v)) > M}
\cdot \frac{d^n(u)}{n\mu_1(p^n)} \cdot \frac{1}{n-1}\, ,
\end{align*}
where we have used that $\sum_{i \in [n]} d^n(i)=n\mu_1(p^n)$. Next, 
\begin{align*}
& \sum_{u \in [n]}\sum_{v \in [n]}
(d^n(u))^2d^n(v)\I{\max(d^n(u),d^n(v)) > M}\\
& \le 
\pran{
\Big(\sum_{u \in [n]: d^n(u)>M} (d^n(u))^2\Big)\cdot \sum_{v \in [n]} d^n(v) + 
\Big(\sum_{v \in [n]:d^n(v) > M} d^n(v)\Big) \cdot \sum_{u \in [n]} (d^n(u))^2}\, ,\\
& = n\mu_1(p^n)\cdot \sum_{u \in [n]: d^n(u)>M} (d^n(u))^2\ +\  n \mu_2(p^n)\cdot \sum_{v \in [n]:d^n(v) > M} d^n(v)\, .
\end{align*}
Since $p^n \to p$, $\mu_1(p^n) \to \mu_1(p)$ and $\mu_2(p_n) \to \mu_2(p)$, for any $\delta > 0$ we may choose $M=M(\delta)$ sufficiently large so that
$\sum_{u \in [n]: d^n(u)>M}(d^n(u))^2\ < \delta n$ and $\sum_{v \in [n]:d^n(v) > M} d^n(v) < \delta n$, for all $n \ge 1$. For such $M$, the previous bound and the two identities which precede it yield that 
\[
\E{\sum_{k,l\ge 0} klA^n(k,l) - \sum_{0 \le k,l\le M} klA^n(k,l) } 
 \le 
\frac{1}{\mu_1(p^n)} (\delta n \mu_1(p^n) + \delta n \mu_2(p^n)). 
\]
By Markov's inequality, it follows that for all $\eps > 0$ there is $M \in \N$ such that for all $n \in \N$, 
\[
\p{\sum_{k,l\ge 0} klA^n(k,l) - \sum_{0 \le k,l\le M} klA^n(k,l) > \eps n} < \eps. 
\]
Finally, since $n^{-1}A^n(k,l) \convp \alpha(k,l)$, it follows that for all $M \in \N$ we have 
\[
\frac{1}{n}\sum_{0 \le k,l\le M} klA^n(k,l) \convp \sum_{0 \le k,l\le M} kl\alpha(k,l)\, ,
\] 
so the preceding probability bound implies that 
\[
\frac{1}{n}\sum_{k,l\ge 0} klA^n(k,l) \convp 
\lim_{M \to \infty}\sum_{0 \le k,l \le M} kl\alpha(k,l) = \sum_{k,l \ge 0} kl\alpha(k,l)\, ,
\]
as required. 
\end{proof}

\section{Proof of Theorem~\ref{Poisson Asymptotics}} \label{app:2}
Let $H$ be a simple graph with vertex set $v(H)=[n]$, and let $G$ be a random graph with degree sequence $\dseq=(d(1),\ldots,d(n))$ sampled according to the configuration model. Recall the definitions of $\cL(G),\cM(G,H),\cN(G,H)$ and $L(G),M(G,H),N(G,H)$ from Section~\ref{sec:poisson_approx}. 
The first subsection will provide a quantitative approximation result for mixed moments of $L,M$ and $N$. In the second subsection, we will use this approximation to prove Theorem~\ref{Poisson Asymptotics}.

\subsection{Deterministic bounds on loops and multi-edges} \label{sec:simp_prob}
Our arguments in this section are based on and fairly closely parallel those from \cite[Chapter 7]{MR3617364}. 
We recall the falling factorial notation $(x)_{\ell} = x(x-1) \dots (x-\ell+1)$. In what follows, it is convenient to  define $(x)_{\ell} = 1$ if $\ell = 0$, and $(x)_{\ell} = 0$ if $\ell < 0$. 

\begin{prop} \label{Moment Bound}
Write $m = \frac{1}{2} \sum_{i=1}^n d(i)$, and write $d_{\max} = \max \{d(1),\dots,d(n)\}$. For any positive integers $q,r,s \in \N$,
\[
\left|\E{(L)_q(M)_r(N)_s} - \frac{(|\cL|)_q(|\cM|)_r(|\cN|)_s}{\prod_{i=0}^{q+2r+s-1} 2m - 1 - 2i} \right|
\leq
C(S_1+S_2)
\]
where $C$ is a constant depending only on $q,r$ and $s$, $S_1$ is defined by the following identity,
\begin{align*}
S_1 \prod_{i=0}^{q+2r+s-1} (2m - 1 - 2i)
=&
(|\cL|)_{q-2}(|\cM|)_r(|\cN|)_s \sum_{1 \leq u \leq n} d(u)^3\\
&+
(|\cL|)_{q-1}(|\cM|)_{r-1}(|\cN|)_s \sum_{1 \leq u \neq v \leq n} d(u)^3 d(v)^2\\
&+
(|\cL|)_{q-1}(|\cM|)_{r}(|\cN|)_{s-1} \sum_{uv \in e(H)} d(u)^2 d(v)\\
&+
(|\cL|)_{q}(|\cM|)_{r-2}(|\cN|)_{s} \sum_{\substack{1 \leq u \leq n \\ u \not\in \{v_1, v_2\}}} d(u)^3 d(v_1)^2 d(v_2)^2\\
&+
(|\cL|)_{q}(|\cM|)_{r-1}(|\cN|)_{s-1} \sum_{\substack{1 \leq u, v_1, v_2 \leq n \\ u,v_1,v_2 \text{distinct} \\ uv_2 \in e(H)}} d(u)^2 d(v_1)^2 d(v_2)\\
&+
(|\cL|)_{q}(|\cM|)_{r}(|\cN|)_{s-2} \sum_{\substack{uv_1,uv_2 \in e(H) \\ v_1 \neq v_2}} d(u) d(v_1) d(v_2),
\end{align*}
and $S_2$ is defined by
\begin{align*}
S_2
&=
(|\cL|)_{q}(|\cN|)_{s} \sum_{k=1}^{r-1} (|\cM|)_{r-k} \sum_{\ell = 0}^{k} d_{\max}^{2\ell} \prod_{i=0}^{q+2r+s-1-(2k-\ell)}\frac{1}{2m-1-2i}.
\end{align*}
\end{prop} 
\begin{proof}
Throughout the proof, write
\[
(x,y,z)
=
\big( (x_1,\dots,x_q),(y_1,\dots,y_r),(z_1,\dots,z_s)\big)
\in
\cL^q \times \cM^r \times \cN^s
\]
to denote a generic element of $\cL^q \times \cM^r \times \cN^s$. For $(x,y,z) \in \cL^q \times \cM^r \times \cN^s$, write
\[
\I{x} = \prod_{i=1}^q\I{x_i},
\ \ \ \ \ \ \
\I{y} = \prod_{i=1}^r\I{y_i},
\ \ \ \ \ \ \ 
\I{z} = \prod_{i=1}^s\I{z_i}.
\]
We say $(x,y,z)$ is \textit{non-repeating} if $x_1,x_2, \dots ,x_q$ are pairwise distinct, $y_1,y_2, \dots ,y_r$ are pairwise distinct, and $z_1,z_2, \dots ,z_s$ are pairwise distinct. In what follows, write
\[
\sumstar := \sum_{\stackrel{(x,y,z) \in \cL^q \times \cM^r \times \cN^s}{(x,y,z) \text{ is non-repeating}}},
\]
and, for $S \subset \cL^q \times \cM^r \times \cN^s$, write
\[
\sumstar_{S} \ := \sum_{\stackrel{(x,y,z) \in S}{(x,y,z) \text{ is non-repeating}}}.
\]
Note that $(L)_q(M)_r(N)_s = \sumstar \ \I{x} \I{y} \I{z}$, so 
\begin{equation} \label{expected value of LMN}
\E{(L)_q(M)_r(N)_s} = \sumstar \ \p{\I{x} \I{y} \I{z} = 1}.
\end{equation}

We say $(x,y,z)$ is \textit{non-conflicting} if the $2q+4r+2s$ half-edges appearing in $x,y$ and $z$ are pairwise distinct, and otherwise we say $(x,y,z)$ is conflicting. Since half-edges in $G$ are paired uniformly at random, for non-conflicting $(x,y,z)$ we have
\begin{equation} \label{non-conflicting}
\p{\I{x} \I{y} \I{z}=1}
=
\prod_{i=0}^{q+2r+s-1} \frac{1}{2m - 1 - 2i}.
\end{equation}

Now, for a given $(x,y,z)$, let $er(x,y,z)$ be defined as follows:
\[
er(x,y,z) 
:=
\p{\I{x}\I{y}\I{z} = 1} - \prod_{i=0}^{q+2r+s-1} \frac{1}{2m - 1 - 2i}.
\]
By (\ref{non-conflicting}), if $(x,y,z)$ is non-conflicting then $er(x,y,z)=0$, so
\begin{align*}
\E{(L)_q(M)_r(N)_s}
&=
\sumstar \
\prod_{i=0}^{q+2r+s-1} \frac{1}{2m - 1 - 2i}
+
\sumstar \
er(x,y,z)\\
&=
\frac{(|\cL|)_q(|\cM|)_r(|\cN|)_s}{\prod_{i=0}^{q+2r+s-1} 2m - 1 - 2i}
+
\sumstar_{(x,y,z) \text{ is conflicting}} er(x,y,z).
\end{align*}
By the triangle inequality this implies
\begin{equation}\label{eq:conflict_bound}
\left| \E{(L)_q(M)_r(N)_s} - \frac{(|\cL|)_q(|\cM|)_r(|\cN|)_s}{\prod_{i=0}^{q+2r+s-1} 2m - 1 - 2i} \right| \leq \sumstar_{(x,y,z) \text{ is conflicting}}
\left| er(x,y,z) \right|.
\end{equation}
To bound the error terms, we must make a distinction between two types of conflicts. This distinction is most easily understood by way of an example. On the one hand, suppose $x_1 = (ui,uj_1)$ and $x_2 = (ui,uj_2)$ for $u \in V(G)$ and distinct $i,j_1,j_2 \in [d(u)]$. Then $\I{x_1} \I{x_2} = 1$ is the event that the half edge $ui$ is joined to $uj_1$ and $uj_2$ simultaneously, and $\p{\I{x_1} \I{x_2} = 1}=0$. On the other hand, suppose $y_1 = (ui_1,vj_1),(ui_2,vj_2)$ and $y_2 = (ui_1,vj_1),(ui_3,vj_3)$ for distinct $u,v \in V(G)$, distinct $i_1,i_2,i_3 \in [d(u)]$, and distinct $j_1,j_2,j_3 \in [d(v)]$. Then $\I{y_1}\I{y_2} = 1$ is the event that $u$ and $v$ are connected by a triple edge, and $\p{\I{y_1} \I{y_2} = 1} > 0$. 

We say a conflicting triple $(x,y,z)$ is a \textit{bad conflict} if $\p{\I{x} \I{y} \I{z}=1}=0$, and otherwise we say $(x,y,z)$ is a \textit{good conflict}. In the above examples, the first is a bad conflict and the second is a good conflict. Let $\cB = \cB (q,r,s) \subseteq \cL^q \times \cM^r \times \cN^s$ and $\cG = \cG (q,r,s) \subseteq \cL^q \times \cM^r \times \cN^s$ be the collections of bad conflicts and good conflicts respectively. The rest of this proof is dedicated to bounding $|\cB|,|\cG|$, and $er(x,y,z)$ for $(x,y,z) \in \cB \cup \cG$. 

\medskip
\noindent
\textit{(Bounding $|\cB|$):} 
If $(x,y,z)$ is a bad conflict then one of the following must hold (in reading the below descriptions, it may be useful to consult Figure~\ref{fig:overlap}):
\begin{enumerate}
\item There exists $1 \leq a < b \leq q$, $u \in V(G)$ and distinct $i_1,i_2,i_3 \in [d(u)]$ such that $x_a = (ui_1,ui_2)$ and $x_b = (ui_1,ui_3)$. Write $\cB_{xx}$ for the set of $(x,y,z) \in \cB$ that contain a pair $(x_a,x_b)$ of this form. 
\item There exists $1 \leq a \leq q$, $1\leq b \leq r$, distinct $u,v \in V(G)$, distinct $i_1,i_2,i_3 \in [d(u)]$, and distinct $j_1,j_2 \in [d(v)]$ such that $x_a = (ui_1,ui_3)$ and $y_b = (ui_1,vj_1),(ui_2,vj_2)$. Write $\cB_{xy}$ for the set of $(x,y,z) \in \cB$ that contain a pair $(x_a,y_b)$ of this form.
\item There exists $1 \leq a \leq q$, $1\leq b \leq s$, $uv \in e(H)$, distinct $i_1,i_2 \in [d(u)]$, and $j \in [d(v)]$ such that $x_a = (ui_1,ui_2)$ and $z_b = (ui_1,vj)$. Write $\cB_{xz}$ for the set of $(x,y,z) \in \cB$ that contain a pair $(x_a,z_b)$ of this form.
\item There exists $1 \leq a < b \leq r$, distinct $u,v_1,v_2 \in V(G)$, distinct $i_1,i_2,i_3 \in [d(u)]$, distinct $j_1,j_2 \in [d(v_1)]$, and distinct $k_1,k_2 \in [d(v_2)]$ such that $y_a = (ui_1,v_1j_1),(ui_2,v_1j_2)$ and $y_b = (ui_1,v_2k_1),(ui_3,v_2k_2)$. Write $\cB_{yy}$ for the set of $(x,y,z) \in \cB$ that contain a pair $(y_a,y_b)$ of this form.
\item There exists $1 \leq a \leq r$, $1 \leq b \leq s$ and distinct $u,v_1,v_2 \in V(G)$ such that $uv_2 \in e(H)$, distinct $i_1,i_2 \in [d(u)]$, distinct $j_1,j_2 \in [d(v_1)]$, and $k \in [d(v_2)]$ such that $y_a = (ui_1,v_1j_1),(ui_2,v_1j_2)$ and $z_b = (ui_1,v_2k)$. Write $\cB_{yz}$ for the set of $(x,y,z) \in \cB$ that contain a pair $(y_a,z_b)$ of this form.
\item There exists $1 \leq a < b \leq s$, distinct $uv_1,uv_2 \in e(H)$, $i \in [d(u)]$, $j \in [d(v_1)]$, and $k \in [d(v_2)]$ such that $z_a = (ui,v_1j)$ and $z_b = (ui,v_2k)$. Write $\cB_{zz}$ for the set of $(x,y,z) \in \cB$ that contain a pair $(z_a,z_b)$ of this form.
\end{enumerate}
\begin{figure}[htb]
\includegraphics[width=0.7\textwidth]{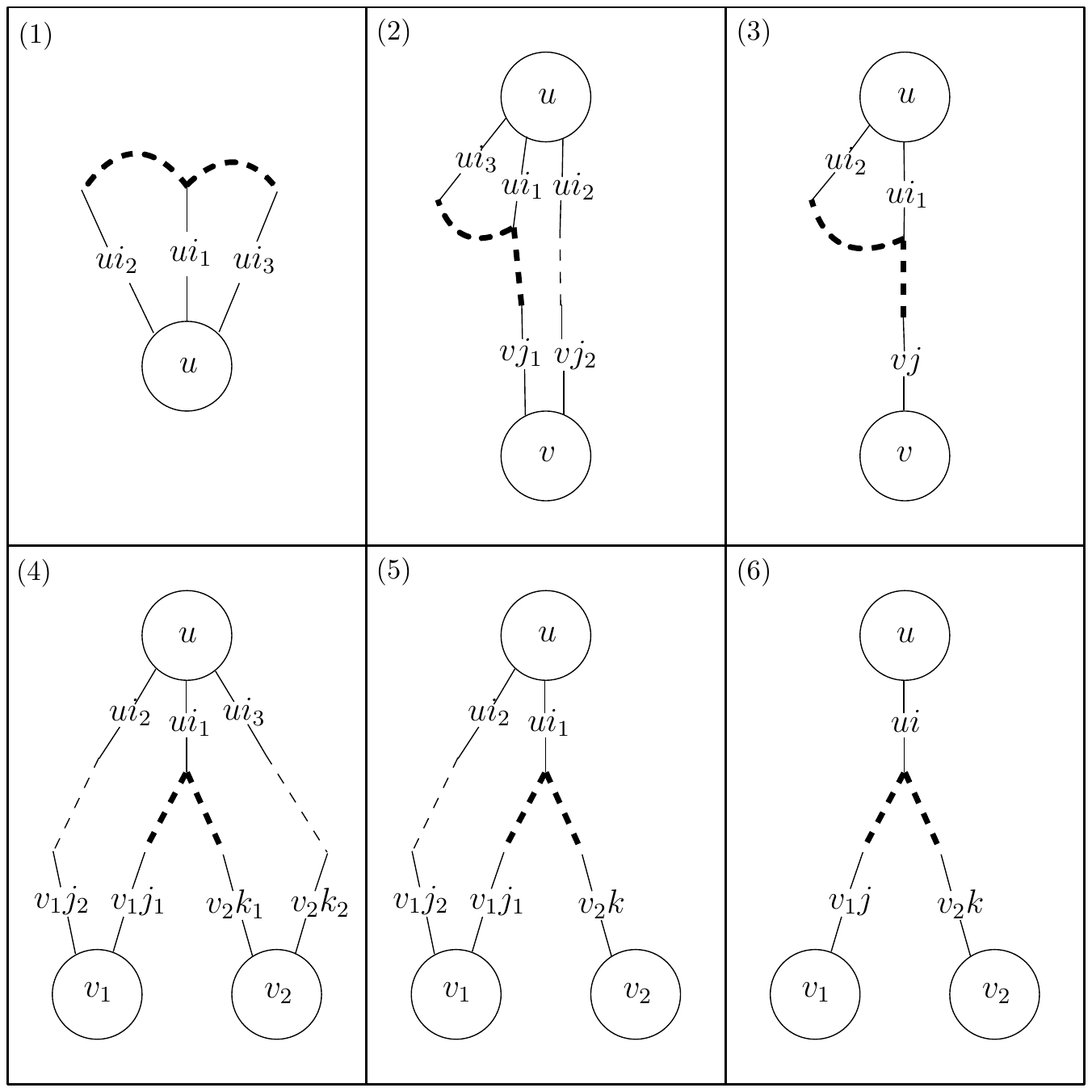}
\caption{An example of each of the six types of ``bad'' conflicts, depicted in the same order as they are described in the text. In the drawing, dashed lines represent the connections between half-edges required by the event. The bold dashed lines show the locations at which two half-edges must pair with a single half-edge, rendering the corresponding event impossible.}
\label{fig:overlap}
\end{figure}
We next turn to bounding the sizes of each set, starting with $\cB_{xx}$. The number of choices for $a$ and $b$ is $q(q-1)/2$. Having chosen these, for each possible choice of $u \in V(G)$, there are less than $d(u)^3$ choices for $i_1,i_2$ and $i_3$. Then, having chosen these, we must choose one of $i_1,i_2$ and $i_3$ to be repeated in $x_a$ and $x_b$. Lastly, we must choose the remaining $q-2$ entries for $x$, $r$ entries for $y$, and $s$ entries for $z$. Hence,
\begin{align*}
|\cB_{xx}|
&\leq
(|\cL|)_{q-2}(|\cM|)_r(|\cN|)_s q(q-1)/2 \sum_{1 \leq u \leq n} 3 d(u)^3\\
&=
C_1 (|\cL|)_{q-2}(|\cM|)_r(|\cN|)_s \sum_{1 \leq u \leq n} d(u)^3,
\end{align*}
where $C_1 = C_1(q) = 3q(q-1)/2$. Note that $\cB_{xx}$ is empty if $q \leq 1$, so $|\cB_{xx}| = 0$. In this case the right hand side is also zero by our convention that $(k)_{\ell} = 0$ for $\ell<0$. Therefore, the bound also holds for $q\leq 1$. The subsequent bounds can likewise be seen to hold when the right hand side is zero, though we do not explicitly verify this in every case.

When building an element of $\cB_{xy}$, the number of ways to choose $a$ and $b$ is $qr$. Having chosen these, for each pair $u,v \in V(G)$, there are less than $d(u)^3 d(v)^2$ choices for $i_1,i_2,i_3 \in [d(u)]$ and $j_1,j_2 \in [d(v)]$. Then, there are a constant number of ways to arrange the half-edges in $x_a$ and $y_b$. Lastly, we must choose the remaining entries for $x,y$ and $z$. Hence,
\begin{align*}
|\cB_{xy}|
&\leq 
C_2 (|\cL|)_{q-1}(|\cM|)_{r-1}(|\cN|)_s \sum_{1 \leq u \neq v \leq n} d(u)^3 d(v)^2,
\end{align*}
where $C_2$ depends only on $q$ and $r$.

For an element of $\cB_{xz}$, for each $uv \in e(H)$, there are less than $d(u)^2 d(v)$ ways to choose $i_1,i_2 \in [d(u)]$ and $j \in [d(v)]$. Hence,
\begin{align*}
|\cB_{xz}|
&\leq
C_3(|\cL|)_{q-1}(|\cM|)_{r}(|\cN|)_{s-1} \sum_{uv \in e(H)} d(u)^2 d(v).
\end{align*}
For $\cB_{yy}$, for each $u,v_1,v_2 \in V(G)$, there are less than $d(u)^3 d(v_1)^2 d(v_2)^2$ ways to choose $i_1,i_2,i_3 \in [d(u)]$, $j_1,j_2 \in [d(v_1)]$, and $k_1,k_2 \in [d(v_2)]$. Hence,
\begin{align*}
|\cB_{yy}|
&\leq
C_4 (|\cL|)_{q}(|\cM|)_{r-2}(|\cN|)_{s} \sum_{\substack{1 \leq u, v_1, v_2 \leq n \\ u,v_1,v_2 \text{ distinct}}} d(u)^3 d(v_1)^2 d(v_2)^2.
\end{align*}
For $\cB_{yz}$, for each $u,v_1,v_2 \in V(G)$ such that $uv_2 \in e(H)$, there are less than $d(u)^2 d(v_1)^2 d(v_2)$ ways to choose $i_1,i_2 \in [d(u)]$, $j_1,j_2 \in [d(v_1)]$, and $k \in [d(v_2)]$. Hence,
\begin{align*}
|\cB_{yz}|
&\leq
C_5 (|\cL|)_{q}(|\cM|)_{r-1}(|\cN|)_{s-1} \sum_{\substack{1 \leq u, v_1, v_2 \leq n \\ u,v_1,v_2 \text{ distinct} \\ uv_2 \in e(H)}} d(u)^2 d(v_1)^2 d(v_2).
\end{align*}
Lastly, for $\cB_{zz}$, for each $uv_1,uv_2 \in e(H)$, there are less than $d(u) d(v_1) d(v_2)$ ways to choose $i \in [d(u)]$, $j \in [d(v_1)]$ and $k \in [d(v_2)]$. Hence,
\begin{align*}
|\cB_{zz}|
&\leq
C_6 (|\cL|)_{q}(|\cM|)_{r}(|\cN|)_{s-2} \sum_{\substack{uv_1,uv_2 \in e(H) \\ v_1 \neq v_2}} d(u) d(v_1) d(v_2).
\end{align*}
Note that the values of $C_1,C_2,C_3,C_4,C_5$ and $C_6$ depend only on $q,r$ and $s$.

\medskip
\noindent
\textit{(Bounding $|er(x,y,z)|$ for $(x,y,z) \in \cB$)}: If $(x,y,z) \in \cB$ then $\p{\I{x} \I{y} \I{z}=1}=0$, meaning 
\begin{equation}\label{eq:baderrbd}
|er(x,y,z)|
=
\prod_{i=0}^{q+2r+s-1} \frac{1}{2m - 1 - 2i}
\end{equation}

\medskip
\noindent
\textit{(Bounding $|\cG|$)}: Suppose $(x,y,z) \in \cG$. Then $(x,y,z)$ is conflicting, meaning a half-edge appears more than once in $x \cup y \cup z$. However, since $\p{\I{x}\I{y}\I{z} = 1} > 0$ for $(x,y,z) \in \cG$, it cannot be the case where $\I{x}\I{y}\I{z}$ contains the event that a half-edge is paired to two different half-edges simultaneously. Hence, there must be a half-edge pair that appears more than once in $x \cup y \cup z$. Furthermore, this half-edge pair must appear more than once in $y$, since the edges in $y$ and $z$ are disjoint. It follows that any $(x,y,z) \in \cG$ can be constructed in the following way:
\begin{enumerate}
\item Choose $x$ and $z$ arbitrarily. The number of choices here is $(|\cL|)_{q}(|\cN|)_{s}$.
\item Choose a set of indices $1 \leq a_1 < a_2 <  \dots < a_k \leq r$ and arbitrarily choose the elements $y_a \in y$ such that $a \neq a_i$ for all $1 \leq i \leq k$. The number of choices here is $(|\cM|)_{r-k}$ times a constant in terms of $r$.
\item Choose $1 \leq b_1 \leq \dots \leq b_{\ell} \leq r$ such that $\{b_1, \dots ,b_{\ell}\} \subseteq \{a_1, \dots ,a_k\}$. For each $1 \leq i \leq \ell$, build $y_{b_i}$ by choosing a half-edge pair already in $y$, then choosing the other half-edge pair arbitrarily. The number of choices for each $i$ is less than $d_{\max}^2$ times a constant in terms of $r$. 
\item For each $a \in \{a_1, \dots ,a_k\} \setminus \{b_1 \dots ,b_{\ell}\}$, build $y_{a}$ by choosing two half-edge pairs already in $y$. The number of choices here is a constant in terms of $r$.
\end{enumerate}
Since every element of $\cG$ can be constructed in this way, we get
\begin{align*}
|\cG|
&\leq
C(r) (|\cL|)_{q}(|\cN|)_{s} \sum_{k=1}^{r-1} (|\cM|)_{r-k} \sum_{\ell = 0}^{k} d_{\max}^{2\ell}.
\end{align*}

\medskip
\noindent
\textit{(Bounding $er(x,y,z)$ for $(x,y,z) \in \cG$)}:
Let $(x,y,z) \in \cG$ and suppose we can construct $(x,y,z)$ as above with a particular $k$ and $\ell$. Then there are $2k - \ell$ half-edge pairs that are redundant when calculating $\p{\I{x} \I{y} \I{z}=1}$. Hence,
\[
\p{\I{x} \I{y} \I{z}=1}
=
\prod_{i=0}^{q+2r+s-1-(2k-\ell)}\frac{1}{2m-1-2i}.
\]
Therefore, for such $(x,y,z) \in \cG$,
\begin{align}
|er(x,y,z)|
&=
\left| \prod_{i=0}^{q+2r+s-1}\frac{1}{2m-1-2i} - \prod_{i=0}^{q+2r+s-1-(2k-\ell)}\frac{1}{2m-1-2i} \right|\nonumber\\
&\leq
\prod_{i=0}^{q+2r+s-1-(2k-\ell)}\frac{1}{2m-1-2i} \, .
\label{eq:gooderrbd}
\end{align}
Finally, by (\ref{eq:conflict_bound}), we know that 
\[
\left| \E{(L)_q(M)_r(N)_s} - \frac{(|\cL|)_q(|\cM|)_r(|\cN|)_s}{\prod_{i=0}^{q+2r+s-1} 2m - 1 - 2i} \right| \leq \sumstar_{(x,y,z) \in \cB}
\left| er(x,y,z) \right|
+ \sumstar_{(x,y,z) \in \cG}
\left| er(x,y,z) \right|,
\]
from which the result now follows by using the bounds on $|\cB_{xx}|,|\cB_{xy}|,|\cB_{xz}|,|\cB_{yy}|,|\cB_{yz}|,|\cB_{zz}|$ and on $|\cG|$, together with (\ref{eq:baderrbd}) and (\ref{eq:gooderrbd})
\end{proof}

\subsection{The probability of simplicity for a random superposition of graphs}

Before proving Theorem~\ref{Poisson Asymptotics}, it will be useful to show some auxiliary bounds. 
\begin{lem} \label{bounds1}
Under the assumptions of Theorem \ref{Poisson Asymptotics}, we have
\begin{equation} \label{sum of degrees is O n}
\sum_{v \in [n]} d^n(v) = O(n),
\end{equation}
\begin{equation} \label{sum of squares is O n}
\sum_{v \in [n]} \left(d^n(v)\right)^2 = O(n),
\end{equation}
\begin{equation} \label{joint distribution is O n}
\sum_{uv \in e(h_n)} d^n(u)d^n(v) = O(n),
\end{equation}
and
\begin{equation} \label{neighbourhoods are o n}
\sup_{u \in [n]} \sum_{v : uv \in e(h_n)} d^n(v) = o(n),
\end{equation}
\end{lem}

\begin{proof}
Equations (\ref{sum of degrees is O n}) and (\ref{sum of squares is O n}) follow from the fact that $\mu_1(p^n) \rightarrow \mu_1(p) < \infty$ and $\mu_2(p^n) \rightarrow \mu_2(p) < \infty$. Indeed, we have
\begin{align*}
\sum_{v \in [n]} d^n(v) 
= 
n \sum_{k \geq 1} k p^n(k)
=
n\mu_1(p^n) = O(n),
\end{align*}
and
\begin{align*}
\sum_{v \in [n]} \left(d^n(v)\right)^2 
= 
n \sum_{k \geq 1} k^2 p^n(k)
=
n\mu_2(p^n) = O(n).
\end{align*}

Equation (\ref{joint distribution is O n}) follows from the convergence of $\sum_{i,j \geq 1} ij\alpha^n(i,j)$. Notice that, by the definition of $\alpha^n$,
\[
\sum_{uv \in e(h_n)} d^n(u) d^n(v)
=
\sum_{i,j \geq 1} \left( \sum_{uv \in e(h_n) : d^n(u)=i, d^n(v)=j} ij \right)
=
\sum_{i,j \geq 1} ij (n\alpha^n(i,j)).
\]
Hence, 
\[
\frac{1}{n} \sum_{uv \in e(h_n)} d^n(u) d^n(v) 
\rightarrow 
\sum_{i,j \geq 1} ij\alpha(i,j) < \infty,
\]
implying that 
\[
\sum_{uv \in e(h_n)} d^n(u) d^n(v) = O(n).
\]

We will prove the fourth bound by contradiction. To this end, suppose (\ref{neighbourhoods are o n}) fails. Then we can find $c > 0$ and a sequence of vertices $(u_n, n \geq 1)$ with $u_n \in v(h_n)$ such that for all $n$ sufficiently large,
\begin{equation} \label{eqn for contradiction proof}
\sum_{v : u_n v \in e(h_n)} d^n(v) \geq cn.
\end{equation}
Write $\deg(u_n) = \deg_{h_n}(u_n)$. List the neighbours of $u_n$ in $h_n$ as $N_{h_n}(u_n) = \{v^n_1,\dots ,v^n_{\deg(u_n)}\}$ so that $d^n(v^n_i) \geq d^n(v^n_{i+1})$ for all $1 \leq i < \deg(u_n)$.

Next, fix $D \in \mathbb{N}$ and let $k = k(n) = \max\{i : d^n(v^n_i) \geq D\}$. Then
\[
\sum_{i = k+1}^{\deg(u_n)} d^n(v^n_i) \leq (\deg(u_n)-k)(D-1) = o(n),
\]
the last bound holding since $\deg(u_n) = o(n)$ by assumption. Thus,
\begin{align*}
\sum_{v : u_n v \in e(h_n)} d^n(v)^2
&=
\sum_{i=1}^{\deg(u_n)} d^n(v^n_i)^2\\
&\geq 
\sum_{i=1}^k d^n(v^n_i)^2\\
&\geq
D \sum_{i=1}^k d^n(v^n_i)\\
&=
D \left(\sum_{v : u_n v \in e(h_n)} d^n(v) - o(n) \right)\\
&\geq
D(c-o(1))n,
\end{align*} 
the last bound holding by (\ref{eqn for contradiction proof}). Since $D \in \mathbb{N}$ was arbitrary, it follows that 
\[
\sum_{v \in [n]} d^n(v)^2
\geq 
\sum_{v : u_nv \in e(h_n)} d^n(v)^2
=
\omega(n),
\]
contradicting (\ref{sum of squares is O n}).
\end{proof}

The next lemma is the last ingredient needed, and also assumes $p(0)+p(1)<1$, i.e. an asymptotically non-zero proportion of the degrees in $G_n$ are 2 or greater. We will show later that Theorem \ref{Poisson Asymptotics} is straightforward when $p(0)+p(1)=1$. 

\begin{lem} \label{bounds2}
Under the assumptions of Theorem \ref{Poisson Asymptotics}, suppose additionally that $p(0)+p(1)<1$. Then
\begin{equation} \label{cL is theta n}
|\cL(G_n)| = \Theta(n),
\end{equation}
and
\begin{equation} \label{cM is theta n^2}
|\cM(G_n,h_n)| = \Theta(n^2).
\end{equation}
\end{lem}

\begin{proof}
Let $\cL_n = \cL(G_n), \cM_n = \cM(G_n,h_n),$ and $\cN_n = \cN(G_n,h_n)$. By their definitions, we have
\begin{align*}
|\cL_n|
&=
\sum_{v \in [n]} \frac{d^n(v) \left( d^n(v) - 1 \right)}{2}, \text{ and}\\
|\cM_n|
&=
\sum_{\{u < v : uv \notin e(h_n)\}} \frac{d^n(u) \left( d^n(u) - 1 \right)}{2} d^n(v) \left( d^n(v) - 1 \right).
\end{align*}

For the upper bounds, by Lemma \ref{bounds1} we have
\begin{align*}
|\cL_n|
&=
\sum_{v \in [n]} \frac{d^n(v) \left( d^n(v) - 1 \right)}{2}
\leq
\sum_{v \in [n]} d^n(v)^2
=
O(n), \text{ and}\\
|\cM_n|
&=
\sum_{\{u < v : uv \notin e(h_n)\}} \frac{d^n(u) \left( d^n(u) - 1 \right)}{2} d^n(v) \left( d^n(v) - 1 \right)
\leq
\left(\sum_{v \in [n]} d^n(v)^2\right)^2
= 
O(n^2).
\end{align*}
 
For the lower bounds, first notice that
\begin{align*}
|\cL_n|
&=
\sum_{v \in [n]} \frac{d^n(v) \left( d^n(v) - 1 \right)}{2}\\
&=
\sum_{v \in [n] : d^n(v) > 1} \frac{d^n(v) \left( d^n(v) - 1 \right)}{2}\\
&\geq 
|\{v \in [n] : d^n(v) > 1\}|\\
&=
n(1-p^n(0)-p^n(1))
\end{align*}
Since $p^n(0)+p^n(1) \rightarrow p(0)+p(1) < 1$ by assumption, this implies $|\{v \in [n] : d^n(v) > 1\}| = \Theta(n)$, and so
\[
|\cL_n| \geq |\{v \in [n] : d^n(v) > 1\}| = \Theta(n).
\] 
Similarly, we have 
\[
|\cM_n| \geq \big| \{ (u,v) : u<v, uv \notin e(h_n) \text{ and } d^n(u), d^n(v) > 1 \} \big|.
\]
From the conditions of Theorem~\ref{Poisson Asymptotics} we know that $\max_{v \in [n]} \{\deg_{h_n}(v)\} = o(n)$, which implies that $|e(h_n)| = o(n^2)$. Hence, we have
\begin{align*}
&\big| \{ (u,v) : u<v, uv \notin e(h_n) \text{ and } d^n(u), d^n(v) > 1 \} \big|\\
=&
\big| \{ (u,v) : u<v, d^n(u), d^n(v) > 1 \} \big| - o(n^2),
\end{align*}
and writing $k = k(n) = |\{v \in [n] : d^n(v) > 1\}|$, we have 
\begin{align*}
\big| \{ (u,v) : u<v, d^n(u), d^n(v) > 1 \} \big| = {k \choose 2} = \Theta(k^2) = \Theta(n^2),
\end{align*}
and hence, $|\cM_n| = \Theta(n^2)$. 
\end{proof}
\begin{proof}[Proof of Theorem~\ref{Poisson Asymptotics}]
Let $m = m(n) = \frac{1}{2} \sum_{v \in [n]} d^n(v)$, and let $q,r$ and $s$ be positive integers. Also, in what follows write $d^n_{\mathrm{max}} = \max_{1 \le i \le n} d^n(i)$; by Fact~\ref{fact:maxdeg} we know that $d^n_{\mathrm{max}}=o(n^{1/2})$. 

Assume for the time being that $p(0)+p(1)<1$ and that $\eta > 0$. Notice that when $\eta > 0$ and $\mu_1(p) > 0$ we have
\[
\frac{|\cN_n|}{n} = \frac{1}{n} \sum_{uv \in e(h_n)} d^n(u)d^n(v) 
=
\sum_{i,j \geq 1} ij\alpha^n(i,j)
\rightarrow 
\sum_{i,j \geq 1} ij\alpha(i,j)
=
\mu_1(p) \eta \in (0,\infty)> 0; 
\]
we have $\mu_1(p) > 0$ since $p(0) < 1$, and $\mu_1(p) < \infty$ since $\mu_2(p) < \infty$. It follows that $|\cN_n| = \Theta(n)$.

We first claim that
\[
\E{(L_n)_q (M_n)_r (N_n)_s} 
= 
\frac{(|\cL_n|)_q(|\cM_n|)_r(|\cN_n|)_s}{\prod_{i=0}^{q+2r+s-1} 2m - 1 - 2i} (1+o(1)).
\]
From Proposition \ref{Moment Bound} we know that
\[
\left| \E{(L_n)_q (M_n)_r (N_n)_s} 
-
\frac{(|\cL_n|)_q(|\cM_n|)_r(|\cN_n|)_s}{\prod_{i=0}^{q+2r+s-1} 2m - 1 - 2i} \right| \leq C(S_1+S_2),
\]
where $S_1$ is defined by the relationship
\begin{align*}
& S_1 \cdot \prod_{i=0}^{q+2r+s-1} (2m - 1 - 2i)\\
&=
(|\cL_n|)_{q-2}(|\cM_n|)_r(|\cN_n|)_s \sum_{v \in [n]} (d^n(v))^3\\
& \ +
(|\cL_n|)_{q-1}(|\cM_n|)_{r-1}(|\cN_n|)_s \sum_{1 \leq u \neq v \leq n} (d^n(u))^3 (d^n(v))^2\\
& \ +
(|\cL_n|)_{q-1}(|\cM_n|)_{r}(|\cN_n|)_{s-1} \sum_{uv \in e(h_n)} (d^n(u))^2 d^n(v)\\
& \ +
(|\cL_n|)_{q}(|\cM_n|)_{r-2}(|\cN_n|)_{s} \sum_{\substack{1 \leq u,v_1,v_2 \leq n \\ u,v_1,v_2 \text{ distinct}}} (d^n(u))^3 (d^n(v_1))^2 (d^n(v_2))^2\\
& \ +
(|\cL_n|)_{q}(|\cM_n|)_{r-1}(|\cN_n|)_{s-1} \sum_{\substack{1 \leq u, v_1, v_2 \leq n \\ u,v_1,v_2 \text{ distinct} \\ uv_2 \in e(h_n)}} (d^n(u))^2 (d^n(v_1))^2 d^n(v_2)\\
& \ +
(|\cL_n|)_{q}(|\cM_n|)_{r}(|\cN_n|)_{s-2} \sum_{\substack{uv_1,uv_2 \in e(h_n) \\ v_1 \neq v_2}} d^n(u) d^n(v_1) d^n(v_2),
\end{align*}
and $S_2$ is defined by
\begin{align*}
S_2
&=
(|\cL_n|)_{q}(|\cN_n|)_{s} \sum_{k=1}^{r-1} (|\cM_n|)_{r-k} \sum_{\ell = 0}^{k} \left(d^n_{\max}\right)^{2\ell} \prod_{i=0}^{q+2r+s-1-(2k-\ell)}\frac{1}{2m-1-2i};
\end{align*}
recall that we set $(k)_{\ell} = 0$ if $\ell < 0$. We now show that
\begin{equation} \label{S1 and S2 bounds}
S_1 = o\left( \frac{(|\cL_n|)_{q}(|\cM_n|)_{r}(|\cN_n|)_{s}}{\prod_{i=0}^{q+2r+s-1} (2m-1-2i)} \right) \text{ and }
S_2 = o\left( \frac{(|\cL_n|)_{q}(|\cM_n|)_{r}(|\cN_n|)_{s}}{\prod_{i=0}^{q+2r+s-1} (2m-1-2i)} \right).
\end{equation}
We will start by bounding $S_1$. Since $|\cL_n|,|\cM_n|,|\cN_n| \rightarrow \infty$ as $n \rightarrow \infty$, to prove the first bound in (\ref{S1 and S2 bounds}) it suffices to establish the following bounds:
\begin{align} 
\nonumber
\sum_{v \in [n]} (d^n(v))^3
&= 
o(|\cL_n|^2),\\ 
\nonumber
\sum_{1 \leq u \neq v \leq n} (d^n(u))^3 (d^n(v))^2
&=
o(|\cL_n||\cM_n|),\\ 
\label{6 bounds on sum of degrees}
\sum_{uv \in e(h_n)} (d^n(u))^2 d^n(v)
&=
o(|\cL_n||\cN_n|),\\ 
\nonumber
\sum_{\substack{1 \leq u, v_1,v_2 \leq n \\ u,v_1,v_2 \text{ distinct}}} (d^n(u))^3 (d^n(v_1))^2 (d^n(v_2))^2
&=
o(|\cM_n|^2),\\ 
\nonumber
\sum_{\substack{1 \leq u, v_1, v_2 \leq n \\ u,v_1,v_2 \text{ distinct} \\ uv_2 \in e(h_n)}} (d^n(u))^2 (d^n(v_1))^2 d^n(v_2)
&=
o(|\cM_n||\cN_n|), \text{ and}\\ 
\nonumber
\sum_{\substack{uv_1,uv_2 \in e(h_n) \\ v_1 \neq v_2}} d^n(u) d^n(v_1) d^n(v_2)
&=
o(|\cN_n|^2). 
\end{align}

Using Lemmas \ref{bounds1} and \ref{bounds2}, together with the fact that $d^n_{\max} = o(n^{1/2})$, we get the following results:
\begin{align*}
\sum_{v \in [n]} (d^n(v))^3
\leq
d^n_{\max} \sum_{v \in [n]} (d^n(v))^2
=
o(n^{1/2}) \cdot O(n)
=
o(n^2)
=
o(|\cL_n|^2),
\end{align*}

\begin{align*}
\sum_{1 \leq u \neq v \leq n} (d^n(u))^3 (d^n(v))^2
&\leq
d^n_{\max} \sum_{1 \leq u \neq v \leq n} (d^n(u))^2 (d^n(v))^2\\
&\leq
d^n_{\max} \left(\sum_{v \in [n]} (d^n(v))^2\right)^2
=
o(n^3)
=
o(|\cL_n||\cM_n|),
\end{align*}

\begin{align*}
\sum_{uv \in e(h_n)} (d^n(u))^2 d^n(v)
\leq
d^n_{\max} \sum_{uv \in e(h_n)} d^n(u) d^n(v)
=
d^n_{\max} |\cN_n|
=
o(|\cL_n||\cN_n|),
\end{align*}

\begin{align*}
\sum_{\substack{1 \leq u, v_1,v_2 \leq n \\ u,v_1,v_2 \text{ distinct}}} (d^n(u))^3 (d^n(v_1))^2 (d^n(v_2))^2
&\leq
d^n_{\max} \sum_{\substack{1 \leq u, v_1,v_2 \leq n \\ u,v_1,v_2 \text{ distinct}}} (d^n(u))^2 (d^n(v_1))^2 (d^n(v_2))^2\\
&\leq
d^n_{\max} \left(\sum_{v \in [n]}^n (d^n(v))^2\right)^3
=
o(n^4)
=
o(|\cM_n|^2),
\end{align*}

\begin{align*}
\sum_{\substack{1 \leq u, v_1, v_2 \leq n \\ u,v_1,v_2 \text{ distinct} \\ uv_2 \in e(h_n)}} (d^n(u))^2 (d^n(v_1))^2 d^n(v_2)
& \leq 
\left(d^n_{\max} \right)^3 \sum_{uv \in e(h_n)} d^n(u)d^n(v)\\
& =
o(n^{3/2}) |\cN_n|
=
o(|\cM_n||\cN_n|),
\end{align*}
and
\begin{align*}
\sum_{\substack{uv_1,uv_2 \in e(h_n) \\ v_1 \neq v_2}} d^n(u) d^n(v_1) d^n(v_2)
&=
\sum_{uv \in e(h_n)} d^n(u)d^n(v) \cdot \sum_{w : uw \in e(h_n)} d^n(w)\\
&\leq
\sum_{uv \in e(h_n)} d^n(u)d^n(v) \cdot \left( \sup_{u \in [n]} \sum_{w : uw \in e(h_n)} d^n(w) \right)\\
&=
|\cN_n| \cdot o(n)\\
&=
o(|\cN_n|^2),
\end{align*}
the last bound holding as $\cN_n = \Theta(n)$.

To prove the bound on $S_2$ from (\ref{S1 and S2 bounds}), first notice that 
\begin{align*}
S_2 
&=
(|\cL_n|)_{q}(|\cN_n|)_{s} \sum_{k=1}^{r-1} (|\cM_n|)_{r-k} \sum_{\ell = 0}^{k} (d^n_{\max})^{2\ell} \prod_{i=0}^{q+2r+s-1-(2k-\ell)}\frac{1}{2m-1-2i}\\
&=
\frac{(|\cL_n|)_{q}(|\cN_n|)_{s}}{\prod_{i=0}^{q+s-1} 2m-1-2i} \sum_{k=1}^{r-1} (|\cM_n|)_{r-k} \sum_{\ell = 0}^{k} (d^n_{\max})^{2\ell} \prod_{i=q+s}^{q+2r+s-1-(2k-\ell)}\frac{1}{2m-1-2i}.
\end{align*}
Hence, it suffices to show the following:
\begin{align*}
\sum_{k=1}^{r-1} (|\cM_n|)_{r-k} \sum_{\ell = 0}^{k} (d^n_{\max})^{2\ell} \prod_{i=q+s}^{q+2r+s-1-(2k-\ell)}\frac{1}{2m-1-2i}
&=
o \left( \frac{(|\cM_n|)_{r}}{\prod_{i=q+s}^{q+2r+s-1} 2m-1-2i} \right).
\end{align*}
Since $r$ is fixed, we need only show that for arbitrary $k \in [1,r-1]$ and $\ell \in [0,k]$,
\begin{align*}
(|\cM_n|)_{r-k} (d^n_{\max})^{2\ell} \prod_{i=q+s}^{q+2r+s-1-(2k-\ell)}\frac{1}{2m-1-2i}
&=
o \left( \frac{(|\cM|)_{r}}{\prod_{i=q+s}^{q+2r+s-1} 2m-1-2i} \right).
\end{align*}
By cancelling out some terms, this follows if we can show that
\begin{align*}
(d^n_{\max})^{2\ell}
&=
o \left( \frac{(|\cM_n|-(r-k))_{k}}{\prod_{i=q+2r+s-(2k-\ell)}^{q+2r+s-1} 2m-1-2i} \right).
\end{align*}
Now since $m = \Theta(n)$, $|\cM_n| = \Theta(n^2)$, and $r$ is a constant, this in turn holds, provided that 
\begin{align*}
(d^n_{\max})^{2\ell}
&=
o \left( \frac{n^{2k}}{n^{2k-\ell}} \right)\\
&=
o \left( n^{\ell} \right),
\end{align*}
which holds since $d^n_{\max} = o(n^{1/2})$.

Therefore, 
\[
S_1 + S_2 = o\left( \frac{(|\cL_n|)_q(|\cM_n|)_r(|\cN_n|)_s}{\prod_{i=0}^{q+2r+s-1} 2m - 1 - 2i} \right),
\]
which proves that
\[
\E{(L_n)_q (M_n)_r (N_n)_s} 
= 
\frac{(|\cL_n|)_q(|\cM_n|)_r(|\cN_n|)_s}{\prod_{i=0}^{q+2r+s-1} 2m - 1 - 2i} (1+o(1)).
\]
Furthermore, since $q,r$ and $s$ are fixed, we obtain that 
\begin{align*}
&\frac{(|\cL_n|)_q(|\cM_n|)_r(|\cN_n|)_s}{\prod_{i=0}^{q+2r+s-1} 2m - 1 - 2i}\\
= \ &
\frac{|\cL_n|^q |\cM_n|^r |\cN_n|^s}{(2m)^{q+2r+s}}(1+o(1))\\
= \ &
\left(\frac{|\cL_n|}{\sum_{v \in [n]} d^n(v)}\right)^q \left(\frac{|\cM_n|}{\left(\sum_{v \in [n]} d^n(v)\right)^2}\right)^r\left(\frac{|\cN_n|}{\sum_{v \in [n]} d^n(v)}\right)^s(1+o(1))
\end{align*}

Next, we claim that
\begin{align*}
\frac{|\cL_n|}{\sum_{v \in [n]} d^n(v)} 
&\rightarrow 
\nu/2,
\frac{|\cM_n|}{\left(\sum_{v \in [n]} d^n(v)\right)^2}
\rightarrow
\nu^2/4, \text{ and }
\frac{|\cN_n|}{\sum_{v \in [n]} d^n(v)}
\rightarrow
\eta.
\end{align*}
For the first of these three claims, we have
\begin{align*}
\frac{|\cL_n|}{\sum_{v \in [n]} d^n(v)} 
=
\frac{\frac{1}{2}\sum_{v \in [n]} d^n(v)(d^n(v)-1)}{\sum_{v \in [n]} d^n(v)}
&=
\frac{1}{2} \frac{\sum_{v \in [n]} (d^n(v))^2 - \sum_{v \in [n]} d^n(v)}{\sum_{v \in [n]} d^n(v)}\\
&=
\frac{1}{2} \left( \frac{\mu_2(p^n) - \mu_1(p^n)}{\mu_1(p^n)} \right)
\rightarrow
\nu/2.
\end{align*}
For the second, we have
\begin{align*}
& \frac{|\cM_n|}{\left(\sum_{v \in [n]} d^n(v)\right)^2}\\
= \ &
\frac{\frac{1}{2} \sum_{\{u < v : uv \notin e(h_n)\}} d^n(u) \left( d^n(u) - 1 \right) d^n(v) \left( d^n(v) - 1 \right)}{\left(\sum_{v \in [n]} d^n(v)\right)^2}\\
= \ &
\frac{1}{4 \left(\sum_{v \in [n]} d^n(v)\right)^2} \Bigg( \left[\sum_{v \in [n]} d^n(v) \left( d^n(v) - 1 \right)\right]^2 - \sum_{v \in [n]} \left[d^n(v) \left( d^n(v) - 1 \right) \right]^2\\ 
& \ \ \ \ \ \ \ \ \ \ \ \ \ \ \ \ \ \ \ \ \ \ \ \ \ \ \ \ \ \ \ \ \ \ \ \ \ - \sum_{uv \in e(h_n)} d^n(u) \left( d^n(u) - 1 \right) d^n(v) \left( d^n(v) - 1 \right) \Bigg).
\end{align*}
The second and third terms vanish in the limit since $\mu_1(p) > 0$ and so
\[
\left(\sum_{v \in [n]} d^n(v)\right)^2 = \Omega(n^2),
\]
and
\begin{align*}
\sum_{v \in [n]} \left[d^n(v) \left( d^n(v) - 1 \right) \right]^2
&\leq
(d^n_{\max})^2 \sum_{v \in [n]} (d^n(v))^2
=
(d^n_{\max})^2 n\mu_2(p^n)
=
o(n^2),
\end{align*}
and
\begin{align*}
\sum_{uv \in e(h_n)} d^n(u) \left( d^n(u) - 1 \right) d^n(v) \left( d^n(v) - 1 \right)
&\leq
(d^n_{\max})^2 \sum_{uv \in e(h_n)} d^n(u) d^n(v)\\
&=
(d^n_{\max})^2 |\cN_n|\\
&=
o(n^2).
\end{align*}
Therefore,
\begin{align*}
\frac{|\cM_n|}{\left(\sum_{v \in [n]} d^n(v)\right)^2}
&=
\frac{1}{4} \frac{\left[\sum_{v \in [n]} d^n(v) \left( d^n(v) - 1 \right)\right]^2}{\left(\sum_{v \in [n]} d^n(v)\right)^2} (1+o(1))\\
&\rightarrow
\frac{1}{4} \frac{\left(\mu_2(p) - \mu_1(p)\right)^2}{\left(\mu_1(p)\right)^2}\\
&=
\nu^2/4.
\end{align*}
For the third claim, we have
\begin{align*}
\frac{|\cN_n|}{\sum_{v \in [n]} d^n(v)}
=
\frac{\sum_{uv \in e(h_n)} d^n(u)d^n(v)}{\sum_{v \in [n]} d^n(v)}
&=
\frac{\sum_{i,j \geq 1} ij\alpha^n(i,j)}{\mu_1(p^n)}\\
&\rightarrow
\frac{\sum_{i,j \geq 1} ij\alpha(i,j)}{\mu_1(p)}
=
\eta.
\end{align*}

Therefore,
\begin{align} \label{putting everything together}
\nonumber
& \lim_{n \rightarrow \infty} \E{(L_n)_q (M_n)_r (N_n)_s} \\
= & 
\lim_{n \rightarrow \infty} \frac{(|\cL_n|)_q(|\cM_n|)_r(|\cN_n|)_s}{\prod_{i=0}^{q+2r+s-1} 2m - 1 - 2i} (1+o(1)) \\
\nonumber
= &
\lim_{n \rightarrow \infty} \left(\frac{|\cL_n|}{\sum_{v \in [n]} d^n(v)}\right)^q \left(\frac{|\cM_n|}{\left(\sum_{v \in [n]} d^n(v)\right)^2}\right)^r\left(\frac{|\cN_n|}{\sum_{v \in [n]} d^n(v)}\right)^s (1+o(1))\\
\nonumber
= &
\left( \nu/2 \right)^q \left( \nu^2/4 \right)^r \left( \eta \right)^s.
\end{align}
It then follows, by Theorem 2.6 of \cite{MR3617364}, that the random variables $L_n,M_n$ and $N_n$ converge to independent Poisson random variables with parameters $\nu/2,\nu^2/4$ and $\eta$ respectively. This proves the theorem in the case that $p(0)+p(1)<1$ and $\eta > 0$.

Lastly we will deal with the cases that arise if $p(0)+p(1) = 1$ or if $\eta = 0$. First, if $\eta = 0$ then $\lim_{n \rightarrow \infty} \sum_{i,j \geq 1} ij \alpha^n(i,j) = \sum_{i,j \geq 1} ij \alpha(i,j) = 0$. For any two half-edges $ui$ and $vj$ with $u,v \in [n]$, $i \in [d^n(u)]$ and $j \in [d^n(v)]$, we have $\p{\I{ui,vj}=1} = \frac{1}{2m-1}$ from the definition of the configuration model. So by $(\ref{expected value of LMN})$, since $m = m(n) = n\mu_1(p^n)/2 = \Theta(n)$, we have that 
\begin{align*}
\E{N_n} 
=
\sum_{uv \in e(h_n)} \sum_{i \in [d^n(u)]} \sum_{j \in [d^n(v)]} \p{\I{ui,vj}=1} 
&=
\frac{\sum_{uv \in e(h_n)} d^n(u)d^n(v)}{2m-1}\\
&=
\frac{n}{2m-1} \sum_{i,j \geq 1} ij \alpha^n(i,j)
=
o(1).
\end{align*}
Hence, $\lim_{n \rightarrow \infty} \E{N_n} = 0$. In this case, if $p(0)+p(1)<1$ then a reprise of the argument leading to (\ref{putting everything together}) gives that for all $q,r \geq 1$,
\[
\lim_{n \rightarrow \infty} \E{(L_n)_q (M_n)_r} = (\nu/2)^q(\nu^2/4)^r,
\]
and therefore that $L_n$ and $M_n$ converge to independent Poisson random variables with parameters $\nu/2$ and $\nu^2/4$ respectively.

Lastly, we deal with the case when $p(0)+p(1)=1$. In this case, $\mu_2(p) = \mu_1(p)$, so $\nu = 0$. Furthermore, 
\[
\E{L_n}
= 
\frac{\frac{1}{2} \sum_{v \in [n]} d^n(v)(d^n(v)-1)}{2m-1}
=
\frac{n}{4m-2}(\mu_2(p^n) - \mu_1(p^n)).
\]
Since $\mu_2(p^n) - \mu_1(p^n) \rightarrow \mu_2(p) - \mu_1(p) = 0$, we get that $\lim_{n \rightarrow \infty} \E{L_n} = 0$, and an analogous argument shows that $\lim_{n \rightarrow \infty} \E{M_n} = 0$. In this case, another reprise of the argument leading to (\ref{putting everything together}) gives that for all $s \geq 1$,
\[
\lim_{n \rightarrow \infty} \E{(N_n)_s} = \eta^s,
\]
so $N_n$ is asymptotically Poisson$(\eta)$ distributed. 
\end{proof}

\begin{proof}[Proof of Corollary~\ref{cor:simple_probability}]
First, the fact that $\sum_{k,\ell\ge 0} k\ell \cdot \alpha(k,\ell) < \infty$ appears in (\ref{eq:klsum_finite}), above,  from which it is immediate that $\eta < \infty$. 
Next, let $G_-(\dseq^n)$ be the subgraph of $G(\dseq^n)$ with edge set $\edges(G(\dseq^n))\setminus \edges(T(\dseq^n))$, and let $\dseq^n_-$ be the degree sequence of $G_-(\dseq^n)$,  as defined in Section~\ref{sec:concentration} previous to Proposition~\ref{prop:akl_conv}. 
Finally, write $L_n=L(G_-(\dseq^n))$, $M_n=M(G_-(\dseq^n),T(\dseq^n))$, and $N_n=N(G_-(\dseq^n),T(\dseq^n))$.  Our aim is to apply Theorem~\ref{Poisson Asymptotics},  with $h_n=T(\dseq^n)$ and $G_n=G_-(\dseq^n)=G(\dseq^n)-T(\dseq^n)$. Note that with these choices, we have $\alpha^n(k,l) = n^{-1}A^n(k,l)$, where $A^n(k,l)$ is as in Proposition~\ref{prop:akl_conv}.
Moreover, we have 
\[
\sum_{l \geq 0} A^n(k,l)
= \#\{u \in [n]: d^n_-(u)=k\}\, .
\]

Conditionally given $T(\dseq^n)$, the graph $G_-(\dseq^n)$ is a random graph with degree sequence $\dseq^n_-$. By Proposition~\ref{prop:akl_conv} we know that $\alpha^n(k,l) \convp \alpha(k,l)$ for all $k,l \ge 0$, and that 
\[
\sum_{k,l \geq 0} kl\alpha^n(k,l) \convp \sum_{k,l \geq 0} kl\alpha(k,l).
\]
Moreover, since $\alpha$ defines a probability distribution, it must be that for all k we have
\begin{equation}\label{eq:dminusformula}
p^n_-(k) := \frac{1}{n}\#\{u \in [n]: d^n_-(u)=k\}
= \frac{1}{n} \sum_{l\ge 0} A^n(k,l) = \sum_{l\ge 0} \alpha^n(k,l) \convp \sum_{l \geq 0} \alpha(k,l)\, . 
\end{equation}
Setting $p_-(k) = \sum_{l \geq 0} \alpha(k,l)$, then (\ref{eq:dminusformula}) states that $p^n_-(k) \convp p_-(k)$ for all $k \geq 0$. Moreover, since $p^n_-(k) \leq p^n(k)$, $p^n_-(k) \convp p_-(k)$, and $p^n(k) \convp p(k)$ for all $k \geq 0$, it follows that $\mu_2(p_-) \leq \mu_2(p) < \infty$ and $\mu_2(p^n_-) \convp \mu_2(p_-)$.
From these observations, Fact~\ref{fact:maxdeg} then implies that 
\[
\max_{v \in [n]}\mathrm{deg}_{T(\dseq^n)}(v)\le \max_{v \in [n]}\mathrm{deg}_{G_n}(v) = o(n^{1/2})\, .
\]
Since $\mu_2(p^n_-)\convp \mu_2(p_-)$, we also have $\mu_1(p^n_-) \convp \mu_1(p_-)$; but 
\[
n\mu_1(p^n_-) = \sum_{i=1}^n d^n_-(u) = \left(\sum_{i=1}^n d^n(u)\right) - 2(n-1)\, ,
\]
so $\mu_1(p_-)=\mu_1(p)-2$. Thus, if $\mu_1(p)>2$ then $\mu_1(p_-)>0$, so $p_-(0)<1$. In this case, applying Theorem~\ref{Poisson Asymptotics}, it follows that conditionally given $T(\dseq^n)$, 
\[
\|\mathrm{Dist}(L_n,M_n,N_n) - \mathrm{Poi}(\nu/2)\otimes 
\mathrm{Poi}(\nu^2/4)\otimes \mathrm{Poi}(\eta)\|_{\mathrm{TV}} \convp 0
\]
as $n \to \infty$. 
If $(L,M,N)$ is $\mathrm{Poi}(\nu/2)\otimes 
\mathrm{Poi}(\nu^2/4)\otimes \mathrm{Poi}(\eta)$-distributed, then we have $\p{L=M=N=0}=\exp(-\nu/2- \nu^2/4-\eta)$; 
since 
$G(\dseq^n)$ is simple if and only if $L_n=M_n=N_n=0$, it follows that 
\begin{align*} 
\probC{G(\dseq^n)~\mathrm{simple}} {T(\dseq^n)} 
& = \probC{L_n=M_n=N_n=0}{T(\dseq^n)}\\ 
& \convp \p{L=M=N=0}= \exp(-\nu/2-\nu^2/4-\eta)\, , 
\end{align*} 
as required. 

It remains to treat the case that $\mu_1(p)=2$, which implies that $\mu_1(p_-)=0$ and $p_-(0)=1$. This case requires a separate argument, which is more involved than one might expect. We note immediately that in this situation, $m=(1+o(1))n$ as $n \to \infty$. 

Write $G_n'$ for the graph obtained from $G_-(\dseq^n)$ by removing the edge $\Gamma(\dseq^n)$. Then let $L_n'=L(G_n')$, $ M_n'=M(G_n',T(\dseq^n))$, and $N_n'=N(G_n',T(\dseq^n))$. Then $G_n'$ is simple precisely if $L_n'+M_n'+N_n'=0$. We will first prove that $\E{L_n' + M_n' + N_n'} \convp 0$, then explain how to deal with the root edge.

By Proposition \ref{prop:tree_formula}, we know that the non-root half-edges chosen for $T(\dseq^n)$ form a uniformly random subset $\mathcal{S}^n$ of $\bigcup_{i=1}^n \{i1, \dots , i(d^n(i)-1)\}$ of size $n-1$. The half-edges which are paired to form $G_-(\dseq^n)$ are precisely the edges of 
$\mathcal{U}:= \bigcup_{i=1}^n \{i1, \dots , i(d^n(i)-1)\} \setminus \mathcal{S}^n$, together with the unique unpaired root half-edge of $T(\dseq^n)$. 

Write $\mathcal{U}$ for the set of half-edges paired to form $G_n'$. By the observations of the preceding paragraph, $\mathcal{U}$ is a uniformly random subset of $\bigcup_{i=1}^n \{i1, \dots , i(d^n(i)-1)\}$ of size $2(m-n)$. Moreover, conditionally given $\mathcal{U}$, the pairing of half-edges in $\mathcal{S}$ is uniformly random and independent of $T(\dseq^n)$. Therefore, we can construct $(G(\dseq^n),T(\dseq^n),\Lambda(\dseq^n))$ as follows. First sample a sequence of $m-n$ disjoint pairs of half-edges from $\bigcup_{i=1}^n \{i1, \dots , i(d^n(i)-1)\}$ uniformly at random and join them to form edges; this determines the set $\mathcal{U}$, and all edges of $G_n'$. Next, build $T(\dseq^n)$ via Pitman's additive coalescent applied to $\bigcup_{i=1}^n \{i1, \dots , i(d^n(i)-1)\} \setminus \mathcal{U}$. Finally, pair the root half-edge of $T^n$ with the sole remaining unpaired non-root half-edge. We analyze this construction procedure in order to bound the expected number of loops and multiple edges in $G_n'$.

Let $(h_1,h_2)$ be a half-edge pair chosen for $G_n'$ in the construction procedure just above. Then $h_1$ and $h_2$ are uniform random half-edges chosen from $\bigcup_{i=1}^n \{i1,\ldots,i(d^n(i)-1)\}$, and $(h_1,h_2)$ is a loop if $v(h_1) = v(h_2)$. Hence,
\begin{align*}
\E{L_n}
&=
\left| e(G_n') \right| \p{v(h_1) = v(h_2)}\\
&=
(m-n) \sum_{i=1}^n \p{v(h_1) = v(h_2) = i}\\
&=
(m-n) \sum_{i=1}^n \frac{(d^n(i)-1)(d^n(i)-2)}{(2m-(n-1))(2m-n)}.
\end{align*}
Similarly, edges $(h_1,h_2)$ and $(h_3,h_4)$ form a double edge if $v(h_1) = v(h_3)$ and $v(h_2) = v(h_4)$ or if $v(h_1) = v(h_4)$ and $v(h_2) = v(h_3)$. Hence,
\begin{align*}
& \E{M_n} \\
= \ & \left(\left| e(G_n') \right|\right) \left(\left| e(G_n')) \right| - 1 \right) \sum_{1 \leq i < j \leq n} \frac{4(d^n(i)-1)(d^n(i)-2)(d^n(j)-1)(d^n(j)-2)}{(2m-n+1))(2m-n)(2m-n-1)(2m-n-2)}\\
= \ &
(m-n) (m-n-1) \sum_{1 \leq i < j \leq n} \frac{4(d^n(i)-1)(d^n(i)-2)(d^n(j)-1)(d^n(j)-2)}{(2m-n+1)(2m-n)(2m-n-1)(2m-n-2)}\\
\leq \ &
2\left( (m-n) \sum_{i=1}^n \frac{(d^n(i)-1)(d^n(i)-2)}{(2m-n-1)(2m-n-2)} \right)^2.
\end{align*} 
Since $m=n+o(n)$, we have $\frac{m-n-1}{2m-n-2} = o(1)$. Since also $\sum_{i=1}^n (d^n(i)-1)(d^n(i)-2) \leq \sum_{i=1}^n d^n(i)^2 = O(n) = O(2m-n-1)$, it follows from the two preceding displayed inequalities that $\E{L_n' + M_n'} \convp 0$.

To show that $\E{N_n'} \convp 0$, we will again use Proposition \ref{prop:tree_formula}. Given a set $\cH \subseteq \bigcup_{i=1}^n \{i1,\ldots,i(d^n(i)-1)\}$, write $\cH_i = \cH \cap \{i1, \dots , i(d^n(i)-1)\}$. By (\ref{eq:exchangeability_1}) we know that for any set $\cH$ as above with $|\cH| = n-1$, for all $1 \leq i \leq n-1$, for any $h \in \cH$ and any root half-edge $r$ with $v(r) \neq v(h)$, 
\[
\Cprob{(r_i,s_i) = (r,h)}{\mathcal{S}^n = \cH} = \Cprob{(r_1,s_1) = (r,h)}{\mathcal{S}^n = \cH}.
\]
Moreover, by construction, $T(d^n)$ and $G_n'$ are conditionally independent given $\mathcal{S}^n$, so for any $1 \le i\le n-1$, 
\begin{align*}
& \Cexp{m_{G_n'}(v(r_i)v(s_i))}{\mathcal{S}^n = \cH}\\
= \ &
\Cexp{m_{G_n'}(v(r_1)v(s_1))}{\mathcal{S}^n = \cH}\\
= \ &
\sum_{j=1}^n \sum_{\substack{k=1 \\ k \neq j}}^n \Cprob{v(r_1) = j, v(s_1) = k}{\mathcal{S}^n = \cH} \cdot \Cexp{m_{G_n'}(jk)}{\mathcal{S}^n = \cH}\\.
\end{align*}
Now, by Proposition~\ref{prop:tree_formula}, 
\[
\Cprob{v(r_1) = j, v(s_1) = k}{\mathcal{S}^n = \cH} = 
\frac{\left| \cH_k \right|}{(n-1)^2}. 
\]
Also, 
\begin{align*}
\Cexp{m_{G_n'}(jk)} {\mathcal{S}^n = \cH}
&=
(m-n) \cdot \frac{(d^n(j)-1-|\cH_j|)(d^n(k)-1-|\cH_k|)}{(2m-2(n-1))(2m-2(n-1)-1)}. 
\end{align*} 
The term $m-n$ above accounts for the number of edges of $G_n'$; the fraction is the probability that a uniformly random pair of half-edges from $\left(\bigcup_{l=1}^n \{l1,\ldots,l(d^n(i)-1)\}\right) \setminus \cH $ are incident to vertices $j$ and $k$. Combining these formulas, we then have that for all $1 \le i \le n-1$, 
\begin{align*}
& \Cexp{m_{G_n'}(v(r_i)v(s_i))}{\mathcal{S}^n = \cH}\\
\leq \ & 
\sum_{j=1}^n \sum_{\substack{k=1 \\ k \neq j}}^n \frac{\left| \cH_k \right|}{(n-1)^2} \cdot (m-n) \cdot \frac{(d^n(j)-1-|\cH_j|)(d^n(k)-1-|\cH_k|)}{(2m-2(n-1))(2m-2(n-1)-1)}\\
= \ & O(1)\cdot 
\sum_{k=1}^n \frac{\left| \cH_k \right|}{(n-1)^2} \cdot \frac{(d^n(k)-1-|\cH_k|)}{(2(m-n)+1)} \sum_{\substack{j=1 \\ j \neq k}}^n (d^n(j)-1-|\cH_j|).
\end{align*}
Now, since $\sum_{j=1}^n |\cH_j| = |\cH|=n-1$, we have 
\begin{align*}
\sum_{\substack{j=1 \\ j \neq k}}^n (d^n(j)-1-|\cH_j|)
&=
2m - n - |\cH| - \left( d^n(k) - 1 - \left| \cH_k \right| \right)
\leq
2(m-n)-1\, ,
\end{align*}
and it follows that 
\begin{align*}
\Cexp{m_{G_n'}(v(r_i)v(s_i))}{\mathcal{S}^n = \cH}
=O(1)\cdot 
\sum_{k=1}^n \frac{\left| \cH_k \right| (d^n(k)-1-|\cH_k|)}{2(n-1)^2}.
\end{align*}
Taking expectation over $\mathcal{S}^n$ in this bound, it follows that 
\[
\E{m_{G_n'}(v(r_i)v(s_i))}
\leq
\frac{1}{(n-1)^2} \E{\sum_{k=1}^n \left| \cH_k \right| (d^n(k)-1-|\cH_k|)}.
\]
We now show that $\E{\sum_{k=1}^n \left| \cH_k \right| (d^n(k)-1-|\cH_k|)} = o(n)$. Notice that $\left| \cH_k \right| \leq d^n(k)$ and that $d^n(k)-1-|\cH_k| = d^n_-(k)$, unless $k$ is incident to the unpaired root half-edge, in which case $d^n(k)-1-|\cH_k| = d^n_-(k) - 1 \ge 0$. Letting $r$ be the unpaired root half-edge, it then follows that
\[
\sum_{k=1}^n \left| \cH_k \right| (d^n(k)-1-|\cH_k|) 
\leq
\sum_{k=1}^n \left(d^n(k)\right)^2 \I{d^n_-(k) > 0}.
\]
Since $|\{k: d^n_-(k)>0\}| \le m-n = o(n)$, by the second assertion of Fact~\ref{fact:maxdeg} it follows that $\sum_{k=1}^n d^n(k)^2\I{d^n_-(k) >0} = o(n)$, which combined with the two preceding inequalities yields that 
\[
\E{m_{G_n'}(v(r_i)v(s_i))} = o\left( \frac{1}{n} \right).
\]
Summing over $i$, it follows that 
\[
\E{N_n'} = \sum_{i=1}^n \E{m_{G_n'}(v(r_i)v(s_i))} = o(1). 
\]
At this point we know that $\E{L_n'+M_n'+N_n'} \to 0$. We now how to deduce the same for $L_n+M_n+N_n$. We provide full details only in the case that $m-n \to \infty$, i.e., that the number of edges of $G_-(\dseq^n)$ tends to infinity, and briefly explain the argument in the simpler case that $m-n=O(1)$.

By Proposition~\ref{eq:gtg_dist}, for any pair $(g,t)$ where $g$ is a graph with degree sequence $\dseq^n$ and $t$ is a spanning tree of $g$, we have 
\begin{equation}\label{eq:ignore_root}
\p{(G(\dseq^n),T(\dseq^n))=(g,t)} \propto  \frac{\sum_{\gamma \in e(g-t)}2^{\I{\gamma~\mathrm{is~a~loop}}}\cdot m_{g-t}(\gamma)}{\prod_{i=1}^n 2^{m_{g-t}(ii)} \cdot \prod_{e \in \edges(g)} m_{g-t}(e)!}\, ,
\end{equation}
and for any edge $\gamma$ of $g-t$, 
\[
\Cprob{\Gamma(\dseq^n)=\gamma}{(G(\dseq^n),T(\dseq^n))=(g,t)}
\propto 2^{\I{\gamma~\mathrm{is~a~loop}}}\cdot m_{g-t}(\gamma).
\]
If the graph with edge set $e(g)-(e(t) \cup \gamma)$ is simple, then $g-t$ has at most one loop, and no edge with multiplicity more than two, so 
\[
\sup_{e \in e(g-t)}\Cprob{\Gamma(\dseq^n)=e}{(G(\dseq^n),T(\dseq^n))=(g,t)} \le \frac{2}{|e(g-t)|}. 
\]
In particular, writing 
\[
\mathcal{B}(g,t) = \cL(g)\cup \cN(g,t) \cup \bigcup_{((ui_i,vj_1),(ui_2,vj_2))\in \cM(g,t)} \{(ui_i,vj_1),(ui_2,vj_2)\}\, ,
\]
and recalling that the half-edges comprising $\Gamma(\dseq^n)$ are $\Gamma^-(\dseq^n)$ and $\Gamma^+(\dseq^n)$, 
it follows that 
\begin{align*}
& \Cprob{(\Gamma^-(\dseq^n),\Gamma^+(\dseq^n)) \in \mathcal{B}(G_-(\dseq^n),T(\dseq^n))}{(G(\dseq^n),T(\dseq^n))}\\
& \le \frac{2|\mathcal{B}(G_-(\dseq^n),T(\dseq^n))|}{|e(G_-(\dseq^n))|}\I{G_n'~\mathrm{simple}} + \I{G_n'~\mathrm{not~simple}}\\
& \le \frac{4(L(G_-(\dseq^n))+N(G_-(\dseq^n),T(\dseq^n))+M(G_-(\dseq^n),T(\dseq^n)))}{|e(G_-(\dseq^n))|} \I{G_n'~\mathrm{simple}}+\I{G_n'~\mathrm{not~simple}}\\
& = \frac{4(L_n+M_n+N_n)}{m-(n-1)}+\I{G_n'~\mathrm{not~simple}}\, ;
\end{align*}
the last inequality is not tight unless $M(G_-(\dseq^n),T(\dseq^n))=0$. Now, 
\begin{equation}\label{eq:boostrap_simple}
L_n+M_n+N_n  \le 3(L_n'+M_n'+N_n'+ \I{(\Gamma^-(\dseq^n),\Gamma^+(\dseq^n)) \in \mathcal{B}(G_-(\dseq^n),T(\dseq^n))})\, ;
\end{equation}
this inequality is never tight but it suffices for our purposes. 
Using this bound, and taking expectations in the previous conditional probability bound, we obtain that 
\begin{align*} 
& \p{(\Gamma^-(\dseq^n),\Gamma^+(\dseq^n)) \in \mathcal{B}(G_-(\dseq^n),T(\dseq^n))} \\
& \le \frac{12}{m-(n-1)}\E{L_n'+M_n'+N_n'+1} + \p{G_n'~\mathrm{not~simple}}\, \\
& = \frac{12}{m-(n-1)}\E{L_n'+M_n'+N_n'+1} + \p{L_n'+M_n'+N_n' > 0}\, .
\end{align*}
Since $\E{L_n'+M_n'+N_n'} \to 0$, if $m-(n-1) \to \infty$ it follows that 
\[
\p{(\Gamma^-(\dseq^n),\Gamma^+(\dseq^n)) \in \mathcal{B}(G_-(\dseq^n),T(\dseq^n))} \convp 0
\]
in which case (\ref{eq:boostrap_simple}) and the fact that $L_n'+M_n'+N_n' \convp 0$ together imply that $L_n+M_n+N_n \convp 0$ as required. 

Finally, if $m-n=O(1)$, the fact that $L_n+M_n+N_n \convp 0$ can be seen as follows. Let $h$ be a uniformly random half-edge from $\bigcup_{1 \le i \le n} \{i1,\ldots,i(d^n(i)-1)\}$. 
Then with high probability, $d^n(v(h))=O(1)$. Since $|\cU|=|\bigcup_{1 \le i \le n} \{i1,\ldots,i(d^n(i)-1)\}\setminus \mathcal{S}^n|=2(m-n)=O(1)$, it follows that with high probability $d^n(v(h))=O(1)$ for all edges of $\bigcup_{1 \le i \le n} \{i1,\ldots,i(d^n(i)-1)\}\setminus \mathcal{H}$ and that $v(h)\ne v(h')$ for all distinct $h,h' \in \cU$. This already implies that $L_n+M_n=0$ with probability $1-o(1)$. 
Finally, by considering Pitman's additive coalescent it is not hard to see that for any vertex $v$ with $d^n(v)=O(1)$, the probability that $v$ is the root of $T^n(v)$ or is adjacent to the root is $o(1)$, and that for any two vertices $v,w$ with $d^n(v)=O(1)$ and $d^n(w)=O(1)$, the probability that $v$ and $w$ are adjacent is $o(1)$. (Verifying the assertions of the last sentence in detail is left to the reader.) This immediately implies that $N_n=0$ with probability $1-o(1)$ and so completes the proof. 
\end{proof}

\section*{Acknowledgements} 
Both authors thank Serte Donderwinkel and an anonymous referee for a very careful reading of the paper and for several useful comments.


\small 

\begin{thebibliography}{20}
\providecommand{\natexlab}[1]{#1}
\providecommand{\url}[1]{\texttt{#1}}
\expandafter\ifx\csname urlstyle\endcsname\relax
  \providecommand{\doi}[1]{doi: #1}\else
  \providecommand{\doi}{doi: \begingroup \urlstyle{rm}\Url}\fi

\bibitem[Abraham et~al.(2013)Abraham, Delmas, and Hoscheit]{MR3035742}
Romain Abraham, Jean-Fran\c{c}ois Delmas, and Patrick Hoscheit.
\newblock A note on the {G}romov-{H}ausdorff-{P}rokhorov distance between
  (locally) compact metric measure spaces.
\newblock \emph{Electron. J. Probab.}, 18:\penalty0 no. 14, 21, 2013.
\newblock \doi{10.1214/EJP.v18-2116}.
\newblock URL \url{https://arxiv.org/abs/1202.5464}.

\bibitem[Berestycki et~al.(2017)Berestycki, Laslier, and Ray]{MR3681382}
Nathana\"{e}l Berestycki, Beno\^{\i}t Laslier, and Gourab Ray.
\newblock Critical exponents on {F}ortuin-{K}asteleyn weighted planar maps.
\newblock \emph{Comm. Math. Phys.}, 355\penalty0 (2):\penalty0 427--462, 2017.
\newblock ISSN 0010-3616.
\newblock \doi{10.1007/s00220-017-2933-7}.

\bibitem[Bernardi(2007)]{MR2285813}
Olivier Bernardi.
\newblock Bijective counting of tree-rooted maps and shuffles of parenthesis
  systems.
\newblock \emph{Electron. J. Combin.}, 14\penalty0 (1):\penalty0 Research Paper
  9, 36, 2007.
\newblock URL
  \url{http://www.combinatorics.org/Volume_14/Abstracts/v14i1r9.html}.

\bibitem[Bernardi(2008)]{MR2438581}
Olivier Bernardi.
\newblock Tutte polynomial, subgraphs, orientations and sandpile model: new
  connections via embeddings.
\newblock \emph{Electron. J. Combin.}, 15\penalty0 (1):\penalty0 Research Paper
  109, 53, 2008.
\newblock URL
  \url{http://www.combinatorics.org/Volume_15/Abstracts/v15i1r109.html}.

\bibitem[Bousquet-M\'{e}lou and Courtiel(2015)]{MR3366469}
Mireille Bousquet-M\'{e}lou and Julien Courtiel.
\newblock Spanning forests in regular planar maps.
\newblock \emph{J. Combin. Theory Ser. A}, 135:\penalty0 1--59, 2015.
\newblock ISSN 0097-3165.
\newblock \doi{10.1016/j.jcta.2015.04.002}.
\newblock URL \url{https://arxiv.org/abs/1306.4536}.

\bibitem[Broutin and Marckert(2014)]{MR3188597}
Nicolas Broutin and Jean-Fran\c{c}ois Marckert.
\newblock Asymptotics of trees with a prescribed degree sequence and
  applications.
\newblock \emph{Random Structures Algorithms}, 44\penalty0 (3):\penalty0
  290--316, 2014.
\newblock ISSN 1042-9832.
\newblock \doi{10.1002/rsa.20463}.
\newblock URL \url{https://arxiv.org/abs/1110.5203}.

\bibitem[Ewain~Gwynne and Sheffield(2017)]{gwynne2017}
Jason~Miller Ewain~Gwynne and Scott Sheffield.
\newblock The {T}utte embedding of the mated-{CRT} map converges to {L}iouville
  quantum gravity.
\newblock arXiv:1705.11161 [math.PR], May 2017.
\newblock URL \url{https://arxiv.org/abs/1705.11161}.

\bibitem[Greenhill et~al.(2014)Greenhill, Kwan, and Wind]{MR3177540}
Catherine Greenhill, Matthew Kwan, and David Wind.
\newblock On the number of spanning trees in random regular graphs.
\newblock \emph{Electron. J. Combin.}, 21\penalty0 (1):\penalty0 Paper 1.45,
  26, 2014.
\newblock \doi{10.37236/3752}.
\newblock URL
  \url{https://www.combinatorics.org/ojs/index.php/eljc/article/view/v21i1p45}.

\bibitem[Greenhill et~al.(2017)Greenhill, Isaev, Kwan, and McKay]{MR3645782}
Catherine Greenhill, Mikhail Isaev, Matthew Kwan, and Brendan~D. McKay.
\newblock The average number of spanning trees in sparse graphs with given
  degrees.
\newblock \emph{European J. Combin.}, 63:\penalty0 6--25, 2017.
\newblock ISSN 0195-6698.
\newblock \doi{10.1016/j.ejc.2017.02.003}.
\newblock URL \url{https://arxiv.org/abs/1606.01586}.

\bibitem[Gwynne et~al.(2018)Gwynne, Kassel, Miller, and Wilson]{MR3778352}
Ewain Gwynne, Adrien Kassel, Jason Miller, and David~B. Wilson.
\newblock Active spanning trees with bending energy on planar maps and
  {SLE}-decorated {L}iouville quantum gravity for {$\kappa > 8$}.
\newblock \emph{Comm. Math. Phys.}, 358\penalty0 (3):\penalty0 1065--1115,
  2018.
\newblock ISSN 0010-3616.
\newblock \doi{10.1007/s00220-018-3104-1}.
\newblock URL \url{https://arxiv.org/abs/1603.09722}.

\bibitem[Gwynne et~al.(2020)Gwynne, Holden, and Sun]{MR4126936}
Ewain Gwynne, Nina Holden, and Xin Sun.
\newblock A mating-of-trees approach for graph distances in random planar maps.
\newblock \emph{Probab. Theory Related Fields}, 177\penalty0 (3-4):\penalty0
  1043--1102, 2020.
\newblock ISSN 0178-8051.
\newblock \doi{10.1007/s00440-020-00969-8}.
\newblock URL \url{https://arxiv.org/abs/1711.00723}.

\bibitem[Holden and Sun(2018)]{MR3861296}
Nina Holden and Xin Sun.
\newblock S{LE} as a mating of trees in {E}uclidean geometry.
\newblock \emph{Comm. Math. Phys.}, 364\penalty0 (1):\penalty0 171--201, 2018.
\newblock ISSN 0010-3616.
\newblock \doi{10.1007/s00220-018-3149-1}.
\newblock URL \url{https://arxiv.org/abs/1610.05272}.

\bibitem[Li et~al.(2017)Li, Sun, and Watson]{li2017}
Yiting Li, Xin Sun, and Samuel~S. Watson.
\newblock Schnyder woods, {SLE}(16), and {L}iouville quantum gravity.
\newblock arXiv:1705.03573, May 2017.
\newblock URL \url{https://arxiv.org/abs/1705.03573}.

\bibitem[McKay(1981)]{MR657198}
Brendan~D. McKay.
\newblock Spanning trees in random regular graphs.
\newblock In \emph{Proceedings of the {T}hird {C}aribbean {C}onference on
  {C}ombinatorics and {C}omputing ({B}ridgetown, 1981)}, pages 139--143. Univ.
  West Indies, Cave Hill Campus, Barbados, 1981.

\bibitem[Miller and Sheffield(2019)]{MR4010949}
Jason Miller and Scott Sheffield.
\newblock Liouville quantum gravity spheres as matings of finite-diameter
  trees.
\newblock \emph{Ann. Inst. Henri Poincar\'{e} Probab. Stat.}, 55\penalty0
  (3):\penalty0 1712--1750, 2019.
\newblock ISSN 0246-0203.
\newblock \doi{10.1214/18-aihp932}.
\newblock URL \url{https://arxiv.org/abs/1506.03804}.

\bibitem[Moon(1970)]{MR0274333}
J. W. Moon.
\newblock \emph{Counting labelled trees}, volume 1969 of \emph{From lectures
  delivered to the Twelfth Biennial Seminar of the Canadian Mathematical
  Congress (Vancouver}.
\newblock Canadian Mathematical Congress, Montreal, Que., 1970.

\bibitem[Mullin(1967)]{MR205882}
R. C. Mullin.
\newblock On the enumeration of tree-rooted maps.
\newblock \emph{Canadian J. Math.}, 19:\penalty0 174--183, 1967.
\newblock ISSN 0008-414X.
\newblock \doi{10.4153/CJM-1967-010-x}.
\newblock URL \url{https://doi.org/10.4153/CJM-1967-010-x}.

\bibitem[Pitman(1999)]{MR1673928}
Jim Pitman.
\newblock Coalescent random forests.
\newblock \emph{J. Combin. Theory Ser. A}, 85\penalty0 (2):\penalty0 165--193,
  1999.
\newblock ISSN 0097-3165.
\newblock \doi{10.1006/jcta.1998.2919}.
\newblock URL \url{https://www.stat.berkeley.edu/~pitman/457.pdf}.

\bibitem[van~der Hofstad(2017)]{MR3617364}
Remco van~der Hofstad.
\newblock \emph{Random graphs and complex networks. {V}ol. 1}.
\newblock Cambridge Series in Statistical and Probabilistic Mathematics, [43].
  Cambridge University Press, Cambridge, 2017.
\newblock ISBN 978-1-107-17287-6.
\newblock \doi{10.1017/9781316779422}.
\newblock URL \url{https://www.win.tue.nl/~rhofstad/NotesRGCN.pdf}.

\bibitem[Walsh and Lehman(1972)]{MR314687}
T. Walsh and A.~B. Lehman.
\newblock Counting rooted maps by genus. {II}.
\newblock \emph{J. Combinatorial Theory Ser. B}, 13:\penalty0 122--141, 1972.
\newblock ISSN 0095-8956.
\newblock \doi{10.1016/0095-8956(72)90049-4}.
\newblock URL
  \url{https://www.sciencedirect.com/science/article/pii/0095895672900494}.

\end{thebibliography}
\end{document}